\documentclass[11pt]{amsart}
\usepackage{graphics}             
\usepackage{graphicx}
\usepackage{subcaption}
\usepackage{amsmath, amsfonts, amssymb, amsxtra, mathtools,amsthm} 
\usepackage{multicol} 
\usepackage{enumerate} 
\usepackage{float}
\usepackage{mathrsfs}
\usepackage{xcolor}
\usepackage{ esint }
\usepackage{comment}
\usepackage{booktabs}
\usepackage{tikz, pgfplots} 
\pgfplotsset{compat = 1.15} 
\pgfplotsset{soldot/.style={color=black,only marks,mark=*}} 
\pgfplotsset{holdot/.style={color=black,fill=white,only marks,mark=*}}
\pgfplotsset{mystyle/.append style={axis lines=middle,
							xlabel={$x$},
							ylabel={$y$},
							xticklabel style={xshift=-0.5ex},
							yticklabel style={yshift=-0.5ex},
							xlabel style={right},
							ylabel style={above}, }}
\usepackage{polynom} 
\usepackage{color} 
\usepackage{arydshln} 
\usepackage{hyperref} 
\allowdisplaybreaks
\usetikzlibrary{decorations.pathmorphing,patterns}
\usepackage{pgfplots}
\usepackage{siunitx}
\usepackage[g]{esvect}
\usepackage{caption}
\usepackage{subcaption}
\usepackage{fancyhdr}
 \usepackage{setspace}
   
\usepackage{graphics,graphicx, multicol,hyperref,enumitem,amsmath}
\usepackage[margin=1in]{geometry}
\pgfplotsset{compat=1.18}
\usetikzlibrary{arrows.meta, decorations.markings, calc}

\usetikzlibrary{arrows.meta, decorations.markings, calc}

\newtheorem{proposition}{Proposition}

\newtheorem{theorem}{Theorem}
\newtheorem{corollary}{Corollary}

\newtheorem{remark}{Remark}

\begin{document}
\title{Exit Times for Brownian Motion and Location Detection}
\author{Cole A. Kratz}
\author{Jeffrey J. Langford}
\address{Department of Mathematics, Bucknell University, Lewisburg, Pennsylvania 17837}

\email{jeffrey.langford@bucknell.edu}
\email{cak039@bucknell.edu}

\begin{abstract}
In 2023, Wyman and Xi \cite{WX} asked ``Can you hear your location on a manifold?'' We pose a related question, asking if probabilistic data can be used to recover location within a manifold. More precisely, we ask if location can be recovered from the distribution of exit times of Brownian sample paths. We show that exit time moments determine location up to symmetry for a variety of two-dimensional domains. Beginning with (convex) domains whose boundaries are conic curves, we show that in elliptic, parabolic, and hyperbolic domains, the first two exit time moments are sufficient to detect location up to symmetry. We establish similar results in unbounded wedge domains and equilateral triangular domains. We show that the full sequence of exit time moments determines location in rectangular domains. We end with two location detection results on general planar domains. The key tool driving our work is a ``Poisson hierarchy,'' which establishes a connection between exit time moments of Brownian motion and solutions to a family of PDE problems.
\end{abstract}



\maketitle

\section{Introduction and Background}

\subsection*{Motivation} Our paper is motivated by the famous mathematical question of Kac \cite{K} who in 1966 asked ``Can you hear the shape of a drum?" A drumhead produces its own sounds or vibration frequencies. As the shape of the drumhead changes, the sounds the drumhead makes, or more precisely its vibration frequencies, change. It is reasonable then to wonder whether two drums that sound alike must look alike, and this is indeed the question posed by Kac. In 1992, Gordon, Webb, and Wolpert \cite{GWW} resolved Kac's question by constructing two different drums that produce the same sounds. See Figure \ref{fig:difdrums}  below. In other words, one cannot necessarily hear the shape of a drum.

\begin{figure}[h!]
    \centering
\begin{tikzpicture}[scale=1, thick]
  \draw[fill=blue]
    (0,2) -- (1,3) -- (1,2) -- (3,2) -- (3,1) -- (2,0) -- (2,1) -- (1,1) -- cycle;
  \node at (1.2,-0.5) {Drum $A$};

  \begin{scope}[xshift=5cm]
    \draw[fill=orange]
      (0,3) -- (1,3) -- (1,2) -- (2,2) -- (3,1) -- (2,1) -- (2,0) -- (0,2) -- cycle;
    \node at (1.2,-0.5) {Drum $B$};
  \end{scope}
\end{tikzpicture}
    \caption{Distinct drums $A$ and $B$ that produce the same sounds.}
    \label{fig:difdrums}
\end{figure}
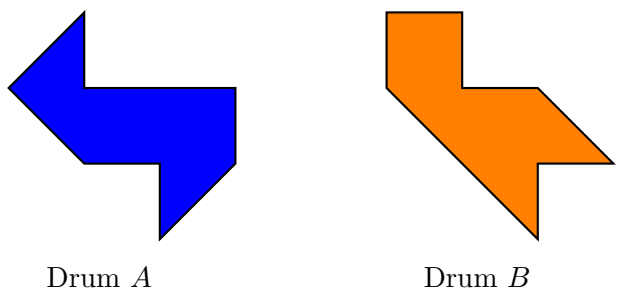 

Although the answer to Kac's question is ultimately `no,' other questions follow naturally. Indeed, Kac's question continues to be a driving influence in the field of spectral geometry and papers on related questions occupy many journal pages. For this paper, the most relevant question comes from Wyman and Xi \cite{WX} who in 2023 asked, ``Can you hear your location on a manifold?'' 
Physically, if you enter a room of known shape while blindfolded, can you determine your location by clapping and listening to the reverberations against the walls? See Figure \ref{fig:wymanxiroom} below.

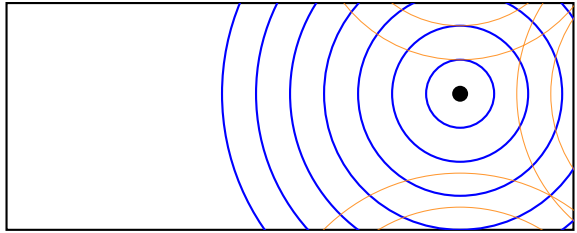
\begin{figure}[H]
    \centering
\begin{tikzpicture}[scale=1.5]
    \draw[thick] (0,0) rectangle (5,2);
    
    \begin{scope}
        \clip (0,0) rectangle (5,2);
        
        \coordinate (Source) at (4, 1.2);
        
        \fill (Source) circle (2pt) node[below left, scale=0.6]{};
        
        \foreach \r in {0.3, 0.6, 0.9, 1.2, 1.5, 1.8, 2.1} {
            \draw[blue, thick, opacity={1.2}] (Source) circle (\r);
        }
        
        
        \foreach \r in {1.2, 1.5} {
            \draw[orange, opacity=0.8] (6, 1.2) circle (\r);
        }
        
        \foreach \r in {1.0, 1.3} {
            \draw[orange, opacity=0.8] (4, 2.8) circle (\r);
        }
        
        \foreach \r in {1.4, 1.7} {
            \draw[orange,  opacity=0.8] (4, -1.2) circle (\r);
        }
    \end{scope}
\end{tikzpicture}
    \caption{Sound reverberations in a rectangular room.}
    \label{fig:wymanxiroom}
\end{figure}

Wyman and Xi answered their question in the affirmative for ellipsoid rooms as well as one- and two-dimensional rectangular rooms with certain boundary conditions. In our paper, we view Wyman and Xi's question through a probabilistic lens. Instead of using sound information, we use Brownian information. The motivating setup and driving question are as follows.

\begin{quote}
 \textbf{\emph{Motivating question:}} Suppose you enter a room of known shape while blindfolded. You have at your disposal a large number of wind-up toys that walk randomly around the room in a Brownian way until they hit a wall. If you track how long it takes for these toys to hit the walls,
can you determine your exact location within the room?
\end{quote}

We use Brownian sample paths to model the behavior of our wind-up toys. Figure \ref{fig:circroompaths} below shows two examples of Brownian paths starting in the middle of a circular room. Notice that the path on the left hits a wall more quickly than the path on the right. The amount of time it takes for a path to first hit a wall is called its exit time. For example, the sample path on the right has a larger exit time than the sample path on the left.
\vspace{1cm}

\begin{figure}[h]
\begin{center}
  \begin{tikzpicture}[scale=1.2, font=\small]

\pgfmathsetseed{2}

\begin{scope}[shift={(6.5,0)}, local bounding box=leftcircle]
  \def\cx{0} \def\cy{0} \def\r{2.5}
  
  \draw[thick, blue] (\cx,\cy) circle (\r);
  
  \xdef\mypathL{(\cx,\cy)}
  \foreach \i in {1,...,2000} {
    \pgfmathsetmacro{\dx}{0.13*rand}
    \pgfmathsetmacro{\dy}{0.13*rand}
    \pgfmathsetmacro{\nx}{\cx + \dx}
    \pgfmathsetmacro{\ny}{\cy + \dy}
    
    \pgfmathsetmacro{\dist}{sqrt((\nx)^2 + (\ny)^2)}
    
    \ifdim \dist pt > \r pt
      \pgfmathsetmacro{\finalx}{\nx * (\r / \dist)}
      \pgfmathsetmacro{\finaly}{\ny * (\r / \dist)}
      \xdef\mypathL{\mypathL -- (\finalx,\finaly)}
      \xdef\cx{\finalx}
      \xdef\cy{\finaly}
      \breakforeach 
    \else
      \xdef\mypathL{\mypathL -- (\nx,\ny)}
      \xdef\cx{\nx}
      \xdef\cy{\ny}
    \fi
  }

  \draw[black, thin, line join=round] \mypathL;
  \filldraw[black] (0,0) circle (1.5pt) ;
  \filldraw[black] (\cx,\cy) circle (1.5pt) ;


\end{scope}

\begin{scope}[shift={(0,0)}, local bounding box=rightcircle]
  \def\cx{0} \def\cy{0} \def\r{2.5}
  
  \draw[thick, blue] (\cx,\cy) circle (\r);

  \xdef\mypathR{(\cx,\cy)}
  \foreach \i in {1,...,2000} {
    \pgfmathsetmacro{\dx}{0.13*rand}
    \pgfmathsetmacro{\dy}{0.13*rand}
    \pgfmathsetmacro{\nx}{\cx + \dx}
    \pgfmathsetmacro{\ny}{\cy + \dy}
    
    \pgfmathsetmacro{\dist}{sqrt((\nx)^2 + (\ny)^2)}
    
    \ifdim \dist pt > \r pt
      \pgfmathsetmacro{\finalx}{\nx * (\r / \dist)}
      \pgfmathsetmacro{\finaly}{\ny * (\r / \dist)}
      \xdef\mypathR{\mypathR -- (\finalx,\finaly)}
      \xdef\cx{\finalx}
      \xdef\cy{\finaly}
      \breakforeach
    \else
      \xdef\mypathR{\mypathR -- (\nx,\ny)}
      \xdef\cx{\nx}
      \xdef\cy{\ny}
    \fi
  }

  \draw[black, thin, line join=round] \mypathR;
  \filldraw[black] (0,0) circle (1.5pt);
  \filldraw[black] (\cx,\cy) circle (1.5pt) ;

\end{scope}

\end{tikzpicture}
\caption{Two Brownian sample paths starting in the middle of a circular room.}
\label{fig:circroompaths}
\end{center}
\end{figure}
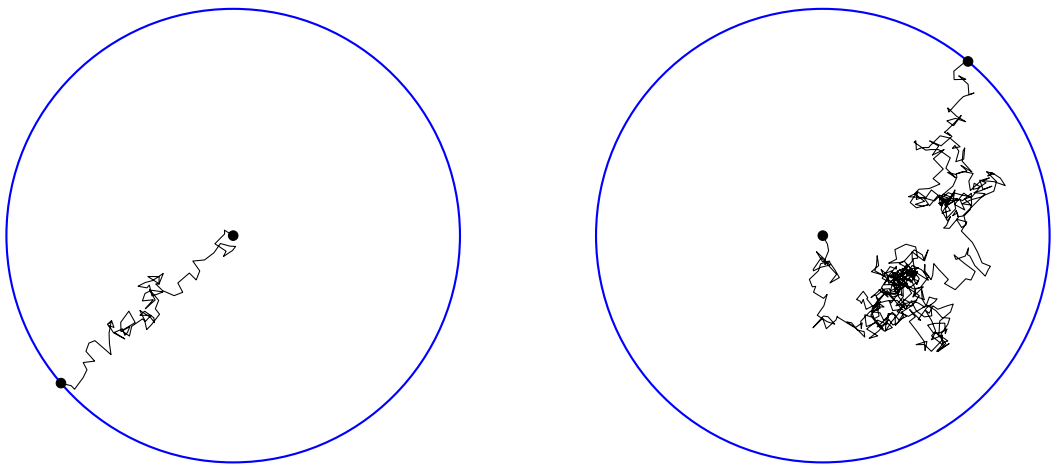

We show that for many room geometries, exit time information is sufficient to determine location up to symmetry. Our results hold for rooms enclosed by an ellipse, parabola, hyperbola, wedge, equilateral triangle, or rectangle.

\subsection*{Main results} To precisely state our results, let $D \subseteq \mathbb{R}^2$ be a planar Lipschitz domain (not necessarily bounded). Let $X_t$ denote Brownian motion on $\mathbb{R}^2$ starting at $(x,y)$. We let $\tau$ denote the first exit time from $D$:
\[
\tau=\tau_D=\inf \{t\geq 0:X_t \notin D\}.
\]
Then $\tau$ is a random variable defined on the space $\Omega$ of Brownian sample paths starting at $(x,y)$. Let $\mathbb{P}^{(x,y)}$ denote the probability measure that charges Brownian paths starting at $(x,y)$. Then for $k\geq 1$, the exit time moments are defined by
\[
\mathbb{E}[\tau^k]=\mathbb{E}^{(x,y)}[\tau_D^k]=\int_{\Omega}\tau_D^k\,d\mathbb{P}^{(x,y)}.
\]
When $k=1$ for instance, $\mathbb{E}^{(x,y)}[\tau_D]$ is the mean exit time and tells us how long it takes on average for a Brownian path starting at $(x,y)$ to leave the domain $D$.

For the majority of the present paper, we focus our attention on various concrete domains $D$. Those domains are listed in the table below.

\begin{center}
\begin{table}[htbp]
\centering
\renewcommand{\arraystretch}{1.6} 
\begin{tabular}{ll}
\toprule
\textbf{Geometry of Domain} & \textbf{Mathematical Definition} \\
\midrule
Elliptic & $\mathcal{E}=\left\{(x,y)\in \mathbb{R}^2 : \left(\frac{x}{a}\right)^2+\left(\frac{y}{b}\right)^2<1\right\} \textup{ for } a,b>0, a\neq b$ \\
Parabolic & $\mathcal{P}=\left\{(x,y)\in \mathbb{R}^2 : y>cx^2\right\} \textup{ for } c>0$ \\
Hyperbolic & $\mathcal{H}=\left\{(x,y)\in \mathbb{R}^2 : \frac{x^2}{a^2}-\frac{y^2}{b^2}>1, x>a\right\} \textup{ for } a,b>0,\ a>(1+\sqrt{2})b$ \\
Wedge & $\mathcal{W}=\left\{z=re^{i\theta}\in \mathbb{R}^2 : 0<r<\infty,\ -\frac{\alpha}{2}<\theta<\frac{\alpha}{2}\right\} \textup{ for } 0<\alpha<\frac{\pi}{4}$ \\
Equilateral Triangle & $\mathcal{T} = \left\{ (x,y) \in \mathbb{R}^2 : -\frac{b}{\sqrt{3}} < x < \frac{b}{\sqrt{3}}, \; 0 < y < b - \sqrt{3}|x| \right\} \textup{ for } b>0$ \\
Rectangular & $\mathcal{R}= \left\{(x,y)\in \mathbb{R}^2 : 0<x<a,\ 0<y<b\right\} \textup{ for } a,b>0, a\neq b$ \\
\bottomrule
\end{tabular}
\caption{Concrete domains considered in our paper.}
\label{tab:domains}
\end{table}
\end{center}
For the non-rectangular domains listed above, we show that the level curves of the first two exit time moments intersect at either a single point, a number of points equal to the number of axes of symmetry, or double that number. Surprisingly, our numerical work suggests that for general rectangular domains, the level curves of the first two exit time moments can intersect at more than four points (see Remark \ref{rmk:numevid}). So in the rectangular setting, we establish a similar result by intersecting the level curves of all of the exit time moments. More precisely we have the following.

\begin{theorem}[Intersection of Moment Level Curves]\label{thm:main}
Let $D$ denote any of the non-rectangular domains from Table \ref{tab:domains} and suppose the level curves for the first two exit time moments $\mathbb{E}^{(x,y)}[\tau_D]=c_1$ and $\mathbb{E}^{(x,y)}[\tau_D^2]=c_2$ intersect for some positive constants $c_1$ and $c_2$.
\begin{itemize}
\item[1.] If $D=\mathcal E$, then the level curves intersect at either four points, two points, or one point determined by $c_1$ and $c_2$.
\item[2.] If $D=\mathcal P, \mathcal H, \mathcal W$, then the level curves intersect at either two points or a single point determined by $c_1$ and $c_2$.
\item[3.] If $D=\mathcal T$, then the level curves intersect at either six points, three points, or a single point determined by $c_1$ and $c_2$.
\end{itemize}

If $D=\mathcal R$ is a rectangular domain as in Table \ref{tab:domains} and $c_k$ for $k\geq 1$ are positive numbers such that the level curves $\mathbb{E}^{(x,y)}[\tau_{\mathcal R}^k]=c_k$ intersect, then their intersection is either one, two, or four points determined by the values $c_k$.
\end{theorem}

Knowing how the level curves of the exit time moments intersect, we return to our motivating question from above. For one of our non-rectangular domains, suppose we are given values $c_1$ and $c_2$ for the first two exit time moments at some unknown location $(x,y)$. According to Theorem \ref{thm:main}, we can reduce the number of possible locations to a finite number of explicitly known points. Moreover, since the variance of the exit time satisfies
\[
\textup{Var}^{(x,y)}[\tau]=\mathbb E^{(x,y)}[\tau^2]-\left(\mathbb E^{(x,y)}[\tau]\right)^2,
\]
knowing values for the first two exit time moments is equivalent to knowing values for the mean exit time and the variance of the exit time. We therefore have the following.

\begin{corollary}[Location Detection]\label{cor:main} Let $D$ be a non-rectangular domain from Table \ref{tab:domains}. Given values of the first two moments of the exit time at some unknown point $(x,y)$ in $D$, one can detect the location $(x,y)$ up to symmetry. Equivalently, given the mean exit time and the variance of the exit time at some unknown point $(x,y)$ in $D$, one can detect the location $(x,y)$ up to symmetry.

If $D=\mathcal R$ is a rectangular domain as in Table \ref{tab:domains} and values of the full sequence of exit time moments are given at some unknown point $(x,y)$ in $\mathcal R$, one can detect the location $(x,y)$ up to symmetry.
\end{corollary}

We end our paper by establishing a number of location detection results on general domains. We show that if a certain collection of orthogonal projections separates points in a domain, or in a fundamental subdomain, then one can detect location, possibly up to symmetry. We direct the reader to Section \ref{secg:gen} for the precise results.

\subsection*{Methodology and structure}
Associated to the sequence of exit time moments is a hierarchy of PDE problems we refer to as the ``Poisson hierarchy.'' To be precise, let $D\subseteq \mathbb{R}^2$ be a bounded Lipschitz domain and let $u_k(x,y)=\mathbb{E}^{(x,y)}[\tau_D^k]$ be the $k$th exit time moment. Let $u_0=1$. Then
\begin{equation}\label{eq:poisshier}
\left\{
\begin{aligned}
            -\Delta u_k &= k u_{k-1} \text{ in } D, \\
            u_k &= 0 \text{ on } \partial D. 
\end{aligned}
\right.
\end{equation}
In particular, when $k=1$, one has
       \[
       \left\{
        \begin{aligned}
            -\Delta u_1 &= 1 \text{ in } D, \\
            u_1 &= 0  \text{ on } \partial D. 
        \end{aligned}
        \right.
        \]
See \cite{DLM, KMM, M2}. By standard existence and uniqueness theorems, $u_k(x,y)=\mathbb{E}^{(x,y)}[\tau_D^k]$ is the only solution to \eqref{eq:poisshier}.
For bounded domains then, finding formulas for the exit time moments is equivalent to solving the PDE problems iteratively. If $D$ is unbounded, things become more subtle. If we know that $u_k(x,y)=\mathbb{E}^{(x,y)}[\tau_D^k]$ is everywhere finite, then by H\"older's inequality, $u_{k-1}(x,y)=\mathbb{E}^{(x,y)}[\tau_D^{k-1}]$ is also everywhere finite and \eqref{eq:poisshier} still holds. But there may be other functions that also satisfy the PDE problem. That is, for unbounded domains, we lose uniqueness. Thus, even if we manage to find a function that satisfies \eqref{eq:poisshier}, it is not necessarily the $k$th exit time moment. It is precisely this uniqueness failure that makes it challenging to find formulas for the exit time moments in unbounded domains. Since parabolic domains can be expressed as a nested union of elliptic domains, the parabolic moment formulas in Theorem \ref{thm:main} arise naturally as limits of elliptic moment formulas. In the hyperbolic and wedge settings, no such reduction exists. To handle these geometries, we rely on conformal mappings, the Schwarz reflection principle, the minimum principle for harmonic and superharmonic functions, and properties of entire harmonic functions. We have intentionally taken a direct and analytical approach to our problems, bypassing more specialized machinery (e.g., the theory of submartingales). Not only does this keep the presentation relatively self-contained, but we also believe it makes our paper more accessible and of wider interest to a general audience.

Our paper contributes to a growing body of literature connecting exit time moments and their integrals to geometric analysis. For early foundational work, we refer the reader to \cite{KMM}. For connections to comparison geometry, see \cite{CLM}, \cite{HMP1}, \cite{HMP2}, and \cite{M2}. For connections to spectral geometry, see \cite{BGJ}, \cite{CLM}, \cite{DLM}, and \cite{HMP3}. We also refer the reader to the survey papers \cite{GH} and \cite{M1} and the references therein. We are unaware of any work that explores the connection between exit time moments and location detection. With the present paper, that study is now underway.

The following sections of the paper contain our main results. We prove that exit time information determines location up to symmetry for elliptic (Section \ref{sect:ellipse}), parabolic (Section \ref{sect:parabola}), hyperbolic (Section \ref{sect:hyp}), wedge (Section \ref{sect:wedge}), equilateral triangular (Section \ref{sect:tri}), and rectangular domains (Section \ref{sect:rect}). Section \ref{secg:gen} contains our location detection results on general domains.

\section{Elliptic Domains}\label{sect:ellipse}
Throughout this section, we let $\mathcal E$ denote the elliptic domain
\[
\mathcal E=\left\{(x,y)\in \mathbb{R}^2:\left(\frac{x}{a}\right)^2+\left(\frac{y}{b}\right)^2<1\right\}
\]
for positive real numbers $a$ and $b$. The boundary of $\mathcal E$ is the ellipse $\left(\frac{x}{a}\right)^2+\left(\frac{y}{b}\right)^2=1$. See Figure \ref{fig:ellipse} below.

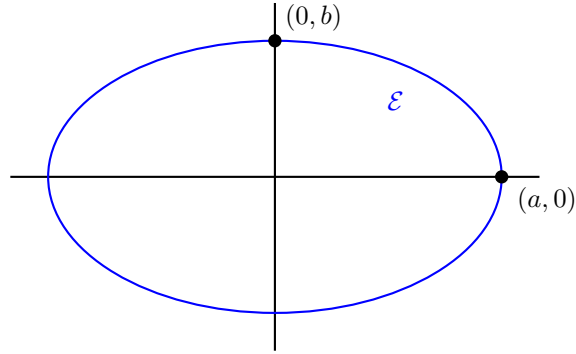
\begin{figure}[H]
\begin{center}
\begin{tikzpicture}[scale=1, font=\small]

    \def\a{3} 
    \def\b{1.8} 
    \pgfmathsetmacro{\c}{sqrt(\a*\a - \b*\b)} 

    \coordinate (Center) at (0,0);
    \coordinate (Top) at (0, \b);
    \coordinate (Bottom) at (0, -\b);
    \coordinate (Right) at (\a, 0);
    \coordinate (Left) at (-\a, 0);
    \coordinate (Focus1) at (-\c, 0);
    \coordinate (Focus2) at (\c, 0);


    \draw[thick, black] (-\a - 0.5, 0) -- (\a + 0.5, 0); 
    \draw[thick, black] (0, -\b - 0.5) -- (0, \b + 0.5); 

    \draw[thick, blue] (Center) ellipse (\a cm and \b cm);

    \node[blue] at (1.6, 1.0) {$\mathcal{E}$};

    
    \coordinate (P) at (45:{\a} and {\b}); 

    \fill (Right) circle (2.5pt) ;
    \fill (Top) circle (2.5pt);
    
    \node[above right] at (Top) {$(0, b)$};
   \node[below right, xshift=2pt] at (Right) {$(a, 0)$};

\end{tikzpicture}
\end{center}
\caption{A picture of the elliptic domain $\mathcal E$ and its boundary.}
\label{fig:ellipse}
\end{figure}

The main result of this section is the following.

\begin{theorem}\label{thm:ellipse}
Let $\mathcal E$ be an elliptic domain as above with $a\neq b$ and let $u_1(x,y)=\mathbb{E}^{(x,y)}[\tau_{\mathcal E}]$ and $u_2(x,y)=\mathbb{E}^{(x,y)}[\tau_{\mathcal E}^2]$  be the first two moments of the exit time from $\mathcal E$. Then
\begin{align*}
u_1(x,y)&=\frac{a^2b^2}{2(a^2 + b^2)}\left(1-\frac{x^2}{a^2} - \frac{y^2}{b^2}\right),\\
u_2(x,y)&= u_1(x,y)(Ax^2 + By^2 + C),
\end{align*}
where
\begin{align*}
    A &= -\frac{b^2(5a^2+b^2)}{6(a^4+6a^2b^2 + b^4)}, \\
    B&= -\frac{a^2(a^2 + 5b^2)}{6(a^4+6a^2b^2 + b^4)}, \\
    C&= \frac{a^2b^2(5a^4 + 26a^2b^2 + 5b^4)}{6(a^2+b^2)(a^4+6a^2b^2 + b^4)}.
\end{align*}
Given positive real numbers $c_1$ and $c_2$, if the level curves $u_1(x,y)=c_1$ and $u_2(x,y)=c_2$ intersect, then their intersection is  either four points, two points, or a single point determined by $c_1$ and $c_2$.
\end{theorem}

\begin{proof}
Direct computation using the given formulas for $u_1(x,y)$ and $u_2(x,y)$ shows that $u_1$ and $u_2$ solve the PDE problems
\begin{equation*}
\left\{
\begin{aligned}
-\Delta u_1 &= 1\ \textup{in }\mathcal E,\\
u_1 &= 0\ \textup{on }\partial \mathcal E,
\end{aligned}
\right.
\qquad\qquad\qquad 
\left\{
\begin{aligned}
-\Delta u_2 &= 2u_1\ \textup{in }\mathcal E,\\
u_2 &= 0\ \textup{on }\partial \mathcal E.
\end{aligned}
\right.
\end{equation*}
Then $u_1(x,y)=\mathbb{E}^{(x,y)}[\tau_{\mathcal E}]$ and $u_2(x,y)=\mathbb{E}^{(x,y)}[\tau_{\mathcal E}^2]$ as discussed in the Introduction (see \eqref{eq:poisshier}).

Next assume that the level curves $u_1(x,y) = c_1$ and $u_2(x,y) = c_2$ intersect for some positive constants $c_1$ and $c_2$. Solving the equation $u_1(x,y) = c_1$ for $x^2$ and $y^2$, we see
\begin{align*}
x^2 &= a^2\left(1 - \frac{2c_1(a^2 + b^2)}{a^2b^2} - \frac{y^2}{b^2}\right),\\
    y^2 &= b^2\left(1 - \frac{2c_1(a^2 + b^2)}{a^2b^2} - \frac{x^2}{a^2}\right).
\end{align*}
Each of the equations above can be substituted back into our equation $u_2(x,y) = c_2$ giving
\begin{equation}\label{eq:ellipsec2c1}
    \frac{c_2}{c_1} = Dx^2 + E = Fy^2 + G,
\end{equation}
for some constants $D,E,F,G$ each depending on $a$ and $b$:
\begin{align*}
    D &= \frac{
   2 b^2 (b^4 - a^4)}{3 (a^2 + b^2) (a^4 + 6 a^2 b^2 + b^4)}, \\
   E &= \frac{(a^2 + 5 b^2)  (2 a^4 b^2 + c_1(a^2 + b^2)^2) }{3 (a^2 + b^2) (a^4 + 6 a^2 b^2 + b^4)}, \\
   F &=  \frac{
   2 a^2 ( a^4 - b^4)}{3 (a^2 + b^2) (a^4 + 6 a^2 b^2 + b^4)}, \\
   G &= \frac{(5a^2 +  b^2)  (2 a^2 b^4  + c_1(a^2 + b^2)^2)}{3 (a^2 + b^2) (a^4 + 6 a^2 b^2 + b^4)}.
\end{align*}
Since $a\neq b$, the constants $D$ and $F$ are nonzero. In particular, we can solve \eqref{eq:ellipsec2c1} explicitly for $x^2$ and $y^2$ giving
\begin{align*}
x^2&=\frac{1}{D}\left(\frac{c_2}{c_1}-E\right),\\
y^2&=\frac{1}{F}\left(\frac{c_2}{c_1}-G\right).
\end{align*}
Since we assumed that the level curves $u_1(x,y) = c_1$ and $u_2(x,y) = c_2$ intersect, the expressions $\frac{1}{D}\left(\frac{c_2}{c_1}-E\right)$ and $\frac{1}{F}\left(\frac{c_2}{c_1}-G\right)$ must be nonnegative.
If $x^2=\frac{1}{D}\left(\frac{c_2}{c_1}-E\right)>0$ and $y^2=\frac{1}{F}\left(\frac{c_2}{c_1}-G\right)>0$, then we have reduced $(x,y)$ to one of four possible locations, reflection symmetric about the $x$- and $y$-axes. If $x^2=\frac{1}{D}\left(\frac{c_2}{c_1}-E\right)>0$ and $y^2=\frac{1}{F}\left(\frac{c_2}{c_1}-G\right)=0$, then we have reduced $(x,y)$ to two possible locations along the $x$-axis. These solutions are reflection symmetric across the $y$-axis. If $x^2=\frac{1}{D}\left(\frac{c_2}{c_1}-E\right)=0$ and $y^2=\frac{1}{F}\left(\frac{c_2}{c_1}-G\right)>0$, we have reduced $(x,y)$ to two possible locations along the $y$-axis. These solutions are reflection symmetric across the $x$-axis.  Finally, if $x^2=\frac{1}{D}\left(\frac{c_2}{c_1}-E\right)=0$ and $y^2=\frac{1}{F}\left(\frac{c_2}{c_1}-G\right)=0$, then $(x,y)=(0,0)$ and we are located at the origin.
\end{proof}

Recall that the variance of the exit time $\textup{Var}^{(x,y)}[\tau]$ satisfies
\[
\textup{Var}^{(x,y)}[\tau]=\mathbb E^{(x,y)}[\tau^2]-\left(\mathbb E^{(x,y)}[\tau]\right)^2.
\]
Thus, assigning values to $\mathbb E^{(x,y)}[\tau]$ and $\mathbb E^{(x,y)}[\tau^2]$ is equivalent to assigning values to $\mathbb E^{(x,y)}[\tau]$ and $\textup{Var}^{(x,y)}[\tau]$. We therefore have the following corollary phrased in the language of location detection.

\begin{corollary}\label{cor:locdetell}
Consider an elliptic domain $\mathcal E$ with $a\neq b$. Given values of the first two moments of the exit time at some unknown point $(x,y)$ in $\mathcal E$, one can detect the location $(x,y)$ up to symmetry. Equivalently, given the mean exit time and the variance of the exit time at some unknown point $(x,y)$ in $\mathcal E$, one can detect the location $(x,y)$ up to symmetry.
\end{corollary}

\section{Parabolic Domains}\label{sect:parabola}

Throughout this section we let $\mathcal P$ denote the parabolic domain
\[
\mathcal P=\{(x,y)\in \mathbb{R}^2:y>cx^2\},
\]
where $c>0$. The boundary of $\mathcal P$ is the parabola $y=cx^2$. See Figure \ref{fig:parabola} below.

\begin{figure}[H]
    \centering
\begin{tikzpicture}[scale=0.6]

  \draw[thick, blue, domain=-4:4, samples=100] plot (\x, {0.5*\x*\x});
  \node[left, blue] at (3.6, 7.76) {$\mathcal{P}$};
    \draw[thick, black] (0, -1) -- (0, 10); 
    \draw[thick, black] (-5, 0) -- (5, 0);

\end{tikzpicture}
\caption{The parabolic domain $\mathcal P$ together with its boundary.}
\label{fig:parabola}
\end{figure}
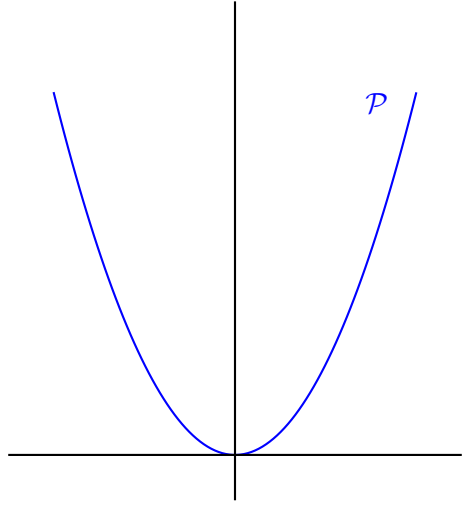

This parabolic setting presents a slight difficulty. Since $\mathcal P$ is unbounded, we cannot directly solve \eqref{eq:poisshier} to compute the exit time moments of $\mathcal P$. To get around this issue, we appeal to our results in elliptic domains. Our strategy is to first show that $\mathcal P$ can be exhausted through a nested union of elliptic domains. A picture of this exhaustion is shown in Figure \ref{fig:ellexh} below. We then use this exhaustion to recover the exit time moments of $\mathcal P$. 

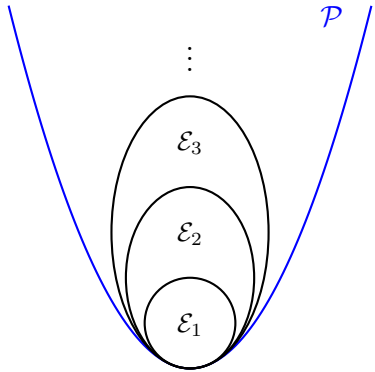
\begin{figure}[H]
\begin{center}
\begin{tikzpicture}[scale=0.6]

  \draw[thick, blue, domain=-4:4, samples=100] plot (\x, {0.5*\x*\x});
  \node[left, blue] at (3.6, 7.76) {$\mathcal{P}$};


  \draw[thick, black] (0,1) ellipse ({sqrt(2)*sqrt(1/2)} and 1);
  \node[black, above] at (0,.5) {\small $\mathcal{E}_1$};

  \draw[thick, black] (0,2) ellipse ({sqrt(2)*sqrt(2/2)} and 2);
  \node[black, above] at (0,2.5) {\small $\mathcal{E}_2$};

  \draw[thick, black] (0,3) ellipse ({sqrt(2)*sqrt(3/2)} and 3);
  \node[black, above] at (0,4.5) {\small $\mathcal{E}_3$};

  \node[black] at (0,7) {$\vdots$};


\end{tikzpicture}
\end{center}
\caption{An illustration of our elliptic exhaustion $\mathcal E_1 \subseteq \mathcal E_2 \subseteq \cdots \subseteq \mathcal P$.}
\label{fig:ellexh}
\end{figure}
\begin{proposition}\label{Prop:exhE2P}
Let $\mathcal E_n$ denote the elliptic domain
\[
\mathcal E_n=\left\{(x,y)\in \mathbb{R}^2:\frac{2cx^2}{n}+\frac{(y-n)^2}{n^2}<1\right\}.
\]
Then $\mathcal E_n \subseteq \mathcal E_{n+1}$ for each $n\geq 1$ and $\bigcup\limits_{n=1}^{\infty} \mathcal E_n=\mathcal P$.
\end{proposition}

\begin{proof}
We first show that $\mathcal E_n \subseteq \mathcal E_{n+1}$. Suppose that $(x,y) \in \mathcal E_n$ so that $\displaystyle \frac{2cx^2}{n}+\frac{(y-n)^2}{n^2}<1$. Note that
\begin{align*}
\frac{2cx^2}{n+1}&=\frac{n}{n+1}\frac{2cx^2}{n}\\
&<\frac{n}{n+1}\left(1-\frac{(y-n)^2}{n^2}\right)\\
&=\frac{2y}{n+1} - \frac{y^2}{n(n+1)}\\
&=\frac{y^2}{(n+1)^2} +1 -\frac{y^2}{n(n+1)}-\frac{(y-(n+1))^2}{(n+1)^2} \\
&=1-\frac{(y-(n+1))^2}{(n+1)^2}+y^2\left(\frac{1}{(n+1)^2}-\frac{1}{n(n+1)} \right)\\
&<1-\frac{(y-(n+1))^2}{(n+1)^2}.
\end{align*}
Therefore,  $\displaystyle \frac{2cx^2}{n+1}+\frac{(y-(n+1))^2}{(n+1)^2}<1$ and $(x,y)\in \mathcal E_{n+1}$ as desired. We next show that $\bigcup\limits_{n=1}^{\infty} \mathcal E_n=\mathcal P$. Suppose that $(x,y)\in \mathcal E_n$ for some $n\geq 1$. Then
\begin{equation}\label{eq:unionEnP}
cx^2<\frac{n}{2}\left(1-\frac{(y-n)^2}{n^2}\right)=\frac{n}{2}\left(-\frac{y^2}{n^2}+\frac{2y}{n} \right)=y-\frac{y^2}{2n} < y.
\end{equation}
Thus $cx^2<y$ and we conclude that $(x,y)\in \mathcal P$. 

Now suppose that $(x,y)\in \mathcal P$ so that $y>cx^2$. Choose $n$ such that $y-\frac{y^2}{2n}>cx^2$. The computation in \eqref{eq:unionEnP} then shows
\begin{align*}
    cx^2&<\frac{n}{2}\left(1-\frac{(y-n)^2}{n^2}\right).
\end{align*} Consequently, 
\begin{equation*}
    \frac{2cx^2}{n} + \frac{(y-n)^2}{n^2} < 1,
\end{equation*}
and $(x,y)\in \mathcal E_n$ as desired. 
\end{proof}

We next use Proposition \ref{Prop:exhE2P} to show that the exit times of our parabolic domain are the pointwise limit of the exit times from elliptic domains using our elliptic exhaustion.

\begin{proposition}\label{prop:limtau}
Let $\tau_{\mathcal P}$ denote the exit time of a Brownian motion from $\mathcal P$. With $\mathcal E_n$ as in Proposition \ref{Prop:exhE2P}, we have
\begin{equation*}
    \lim_{n \to \infty} \tau_{\mathcal E_n} = \tau_{\mathcal P}.
\end{equation*}
\end{proposition}

\begin{proof}
Since $\mathcal E_n \subseteq \mathcal E_{n+1}\subseteq \mathcal P$ for each $n$, we have that
\begin{equation*}
    \tau_{\mathcal E_1}\leq \tau_{\mathcal E_2} \leq \cdots \leq \tau_{\mathcal P}
\end{equation*}
which immediately implies that
\[
\lim_{n\to \infty}\tau_{\mathcal E_n}\leq \tau_{\mathcal P}.
\]
For the reverse inequality, let $X_t$ denote Brownian motion starting at $(x,y)\in \mathcal P$. Let $\Omega$ be the collection of Brownian sample paths starting at $(x,y)$, and let $\mathbb{P}^{(x,y)}$ denote the measure charging Brownian paths starting at $(x,y)$. Suppose that $\omega\in \Omega$ is a realization of a continuous sample path and let $t$ be such that $t < \tau_{\mathcal P}(\omega)$. Then
\[
K=\{X_s(\omega):0\leq s\leq t\}
\]
is a compact set since Brownian sample paths are continuous. Since $K\subseteq\bigcup\limits_{n=1}^{\infty} \mathcal E_n$, we may pass to a finite subcover, and so for some positive integer $N$ we have $K\subseteq \mathcal E_N$. See Figure \ref{fig:EN} below. It follows that
\[
t  <\tau_{\mathcal E_N}(\omega)\leq \tau_{\mathcal P}(\omega).
\]
Since the $\tau_{\mathcal E_n}$ are increasing, it follows that for each $n\geq N$ we have
\[
t  <\tau_{\mathcal E_n}(\omega)\leq \tau_{\mathcal P}(\omega).
\]
Since $t<\tau_{\mathcal P}(\omega)$ was arbitrary, we conclude
\[
    \lim_{n \to \infty} \tau_{\mathcal E_n}(\omega) = \tau_{\mathcal P}(\omega)
\]
as desired.
\end{proof}

\begin{center}
\begin{figure}[H]
    \centering

\begin{tikzpicture}[scale=.75]

  \draw[thick, blue, solid, domain=-6:6, samples=100] plot (\x, {-4 + (2/9)*\x*\x});

  \draw[thick, black] (0,0) ellipse (3 and 4);

  \node[above] at (0,3) {\small$\mathcal E_N$};

  \def\cx{0}
  \def\cy{0}
  \def\tx{0} 
  \def\ty{0} 
  \def\passed{0}
  \def\stopped{0}
  \xdef\mypath{(0,0)}

  \pgfmathsetseed{106} 

  \foreach \i in {1,...,3000} {
    \ifnum\stopped=0
      \pgfmathsetmacro{\dx}{0.1*rand}
      \pgfmathsetmacro{\dy}{0.1*rand}
      \pgfmathsetmacro{\nx}{\cx + \dx}
      \pgfmathsetmacro{\ny}{\cy + \dy}
      
      \ifnum\passed=0
        \pgfmathsetmacro{\val}{(\nx/3)^2 + (\ny/4)^2}
        \pgfmathparse{\val < 0.6 ? 1 : 0}
        \ifnum\pgfmathresult=0
          \xdef\passed{1}
          \xdef\tx{\cx} 
          \xdef\ty{\cy}
        \fi
      \fi
      
      \pgfmathsetmacro{\fnew}{\ny + 4 - (2/9)*\nx*\nx}
      \pgfmathparse{\fnew < 0 ? 1 : 0}
      \ifnum\pgfmathresult=1
        \xdef\stopped{1} 
        
        \pgfmathsetmacro{\fold}{\cy + 4 - (2/9)*\cx*\cx}
        \pgfmathsetmacro{\t}{\fold / (\fold - \fnew)}
        \pgfmathsetmacro{\nx}{\cx + \t*\dx}
        \pgfmathsetmacro{\ny}{\cy + \t*\dy}
      \fi
      
      \xdef\mypath{\mypath -- (\nx,\ny)}
      \xdef\cx{\nx}
      \xdef\cy{\ny}
    \fi
  }

  \draw[black, ultra thin, line join=round] \mypath;

  \filldraw[black] (0,0) circle (1.5pt) node[left] {$(x,y)$};
  
  \filldraw[red] (\tx,\ty) circle (1.5pt) node[below  = 15pt, left = -6 pt] {$X_t(\omega)$};

  \filldraw[blue] (\cx,\cy) circle (1.5pt) node[below right] {$X_{\tau_{\mathcal P}}(\omega)$};
\node[left, blue] at (5.3, 3.46) {$\mathcal{P}$};
\end{tikzpicture}
\caption{The elliptic domain $\mathcal E_N$ from the proof of Proposition \ref{prop:limtau}.}
\label{fig:EN}
\end{figure}
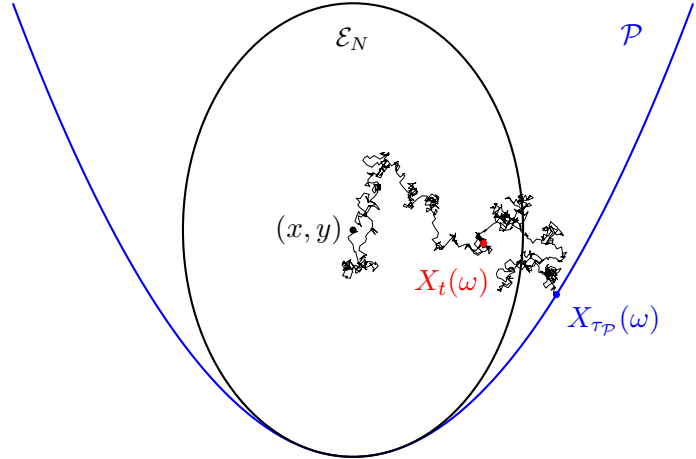
\end{center}
With the previous result in hand, we can concisely express the exit time moments of $\mathcal P$ as limits of the exit time moments of $\mathcal E_n$.
\begin{proposition}\label{Prop:Pformulasu1u2}
Let $\mathcal P$ denote our usual parabolic domain and let $u_1^{\mathcal P}(x,y)=\mathbb{E}^{(x,y)}[\tau_{\mathcal P}]$ and $u_2^{\mathcal P}(x,y)=\mathbb{E}^{(x,y)}[\tau_{\mathcal P}^2]$ denote the first two exit time moments of Brownian motion in $\mathcal P$. Then
\begin{align*}
u_1^{\mathcal P}(x,y)&=\frac{1}{2c}(y-cx^2),\\
u_2^{\mathcal P}(x,y)&=\frac{1}{6c}\left(5y-cx^2+\frac{5}{c}\right)u_1^{\mathcal P}(x,y).
\end{align*}
\end{proposition}

\begin{proof}
Let $\mathcal E_n$ be as in Proposition \ref{Prop:exhE2P}. For $(x,y)\in \mathcal P$ let $X_t$ denote Brownian motion starting at $(x,y)$. As before, let $\Omega$ be the collection of Brownian sample paths starting at $(x,y)$, and let $\mathbb{P}^{(x,y)}$ denote the measure charging Brownian paths starting at $(x,y)$. Since the exit times $\tau_{\mathcal E_n}$ are increasing, Proposition \ref{prop:limtau} and The monotone convergence theorem give
\[
\lim_{n\to \infty}\mathbb{E}^{(x,y)}[\tau_{\mathcal E_n}]=\lim_{n\to \infty}\int_{\Omega}\tau_{\mathcal E_n}\,d\mathbb{P}^{(x,y)} =\int_{\Omega}\tau_{\mathcal P}\,d\mathbb{P}^{(x,y)}=\mathbb{E}^{(x,y)}[\tau_{\mathcal P}].
\]
Denoting $u_1^{\mathcal E_n}(x,y)=\mathbb{E}^{(x,y)}[\tau_{\mathcal E_n}]$, it follows that
\begin{equation}\label{eq:u1EnP}
\lim_{n\to \infty}u_1^{\mathcal E_n}(x,y)=u_1^{\mathcal P}(x,y).
\end{equation}
Denoting $u_2^{\mathcal E_n}(x,y)=\mathbb{E}^{(x,y)}[\tau_{\mathcal E_n}^2]$, we similarly have
\begin{equation}\label{eq:u2EnP}
\lim_{n\to \infty}u_2^{\mathcal E_n}(x,y)=u_2^{\mathcal P}(x,y).
\end{equation}
By Theorem \ref{thm:ellipse}, each $u_1^{\mathcal E_n}(x,y)$ takes the form
\begin{align*}
u_1^{\mathcal E_n}(x,y)&=\frac{\frac{n}{2c}n^2}{2\left(\frac{n}{2c}+n^2\right)}\left(1-\frac{2cx^2}{n}-\frac{(y-n)^2}{n^2}\right)\\
&=\frac{n^3}{4c\left(\frac{n}{2c}+n^2\right)}\left(-\frac{2cx^2}{n}-\frac{y^2}{n^2}+\frac{2y}{n}\right).
\end{align*}
Letting $n \to \infty$ and using \eqref{eq:u1EnP}, we conclude $$u_1^{\mathcal P}(x,y) =  -\frac{x^2}{2}+\frac{y}{2c} = \frac{1}{2c}\left(y-cx^2\right).$$ To compute $u_2(x,y)$, we proceed in a similar fashion. By Theorem \ref{thm:ellipse}, each $u_2^{\mathcal E_n}(x,y)$ takes the form
\[
u_2^{\mathcal E_n}(x,y)=u_1^{\mathcal E_n}(x,y)(A_nx^2+B_n(y-n)^2+C_n),
\]
where
\begin{align*}
A_n&=-\frac{cn(5+2cn)}{3+12cn(3+cn)},\\
B_n&=-\frac{1+10cn}{6+24cn(3+cn)},\\
C_n&=\frac{n^2(5+2cn)(1+10cn)}{6(1+2cn)(1+4cn(3+cn))}.
\end{align*}
Letting $n\to \infty$ and using \eqref{eq:u2EnP}, we have
\[
  u_2^{\mathcal P}(x,y) =u_1^{\mathcal P}(x,y) \frac{1}{6c}\left(5y-cx^2+\frac{5}{c}\right),
\]

as desired.
\end{proof}

With formulas in hand for the first and second exit time moments, we can now prove the main result of this section.

\begin{theorem}\label{thm:parabola}
Let $\mathcal P$ be our parabolic domain and let $u_1(x,y)=\mathbb{E}^{(x,y)}[\tau_{\mathcal P}]$ and $u_2(x,y)=\mathbb{E}^{(x,y)}[\tau_{\mathcal P}^2]$  be the first two moments of the exit time from $\mathcal P$. Given positive real numbers $c_1$ and $c_2$, if the level curves $u_1(x,y)=c_1$ and $u_2(x,y)=c_2$ intersect, then their intersection is either two points or a single point determined by $c_1$ and $c_2$.
\end{theorem}

\begin{proof}
Assuming $u_1(x,y)=c_1$ and $u_2(x,y)=c_2$, our formulas from Proposition \ref{Prop:Pformulasu1u2} give
\[
c_2=\frac{c_1}{6c}\left(5y-cx^2+\frac{5}{c}\right)=\frac{c_1}{6c}\left(2cc_1+4y+\frac{5}{c}\right).
\]
Solving for $y$ gives
\begin{equation}\label{eqn:ypar}
y=\frac{3cc_2}{2c_1}-\frac{cc_1}{2}-\frac{5}{4c}.
\end{equation}
Plugging this formula for $y$ back into $u_1(x,y)=c_1$ and solving for $x^2$ shows
\begin{equation}\label{eq:xsqpar}
x^2=\frac{3c_2}{2c_1}-\frac{5c_1}{2}-\frac{5}{4c^2}.
\end{equation}
Since we have assumed that the level curves $u_1(x,y)=c_1$ and $u_2(x,y)=c_2$ intersect, it must be that $\frac{3c_2}{2c_1}-\frac{5c_1}{2}-\frac{5}{4c^2}\geq 0$. If $\frac{3c_2}{2c_1}-\frac{5c_1}{2}-\frac{5}{4c^2}=0$, then the level curves $u_1=c_1$ and $u_2=c_2$ intersect at a single point $(x,y)$ on the $y$-axis with $y$ determined by \eqref{eqn:ypar}. If $\frac{3c_2}{2c_1}-\frac{5c_1}{2}-\frac{5}{4c^2}>0$ then the level curves intersect at two points symmetric about the $y$-axis, found by solving \eqref{eq:xsqpar} for $x$ (yielding two $x$-values) and by using \eqref{eqn:ypar} to get $y$.
\end{proof}

Analogous to Corollary \ref{cor:locdetell}, we have the following corollary on location detection.

\begin{corollary}\label{cor:locdetpar}
Consider a parabolic domain $\mathcal P$. Given values of the first two moments of the exit time at some unknown point $(x,y)$ in $\mathcal P$, one can detect the location $(x,y)$ up to symmetry. Equivalently, given the mean exit time and the variance of the exit time at some point $(x,y)$ in $\mathcal P$, one can detect the location $(x,y)$ up to symmetry.
\end{corollary}

\section{Hyperbolic Domains}\label{sect:hyp}

Throughout this section, we let $\mathcal H$ denote the hyperbolic domain
\[
\mathcal H=\left\{(x,y)\in \mathbb{R}^2:\frac{x^2}{a^2}-\frac{y^2}{b^2}>1 \textup{ and }x>a\right\}
\]
where $a$ and $b$ are positive real numbers. The boundary of $\mathcal H$ is the right branch of the hyperbola $\frac{x^2}{a^2}-\frac{y^2}{b^2}=1$. See Figure \ref{fig:pichyperbola} below.

\begin{figure}[H]
\begin{center}
\begin{tikzpicture}[scale=.75, font=\small]

    \def\a{3} 
    \def\b{1.8} 
    \pgfmathsetmacro{\c}{sqrt(\a*\a + \b*\b)} 

    \coordinate (Center) at (0,0);
    \coordinate (Right) at (\a, 0);
    \coordinate (Left) at (-\a, 0);
    \coordinate (Focus1) at (-\c, 0);
    \coordinate (Focus2) at (\c, 0);


    \draw[thick, black] (-5.5, 0) -- (5.5, 0); 
    \draw[thick, black] (0, -3.5) -- (0, 3.5); 

    \draw[thick, blue, samples=100] plot[domain=-1.5:1.5] ({\a*cosh(\x)}, {\b*sinh(\x)});
    \draw[thick, blue, samples=100] plot[domain=-1.5:1.5] ({-\a*cosh(\x)}, {\b*sinh(\x)});

    \node[blue] at (4.7, 1.5) {$\mathcal{H}$};


    \fill (Right) circle (2.5pt);
    \fill (Left) circle (2.5pt);
    
    \node[below right, xshift=2pt] at (Right) {$(a, 0)$};
    \node[below left, xshift=-2pt] at (Left) {$(-a, 0)$};

\end{tikzpicture}
\end{center}
\caption{A picture of the hyperbolic domain $\mathcal H$. The right branch of the hyperbola is $\partial \mathcal H$.}
\label{fig:pichyperbola}
\end{figure}
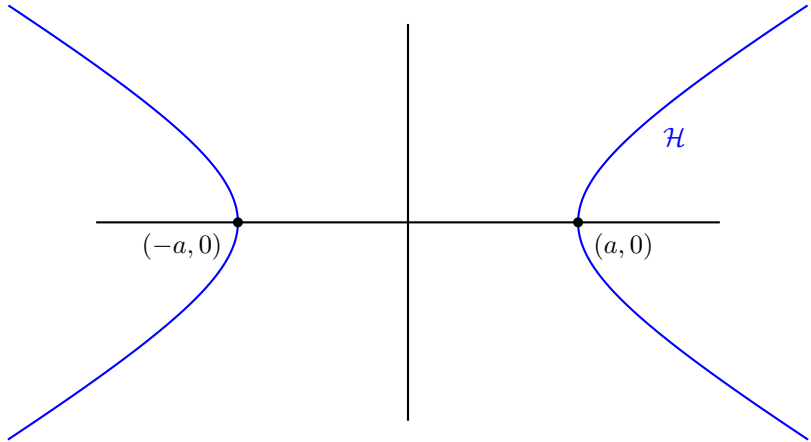

Our first result in this section gives formulas for the first and second exit time moments in a hyperbolic domain.

\begin{proposition}\label{prop:momformshyp}
Let $\mathcal H$ denote the hyperbolic domain above and let $u_1(x,y)=\mathbb{E}^{(x,y)}[\tau_{\mathcal H}]$ and $u_2(x,y)=\mathbb{E}^{(x,y)}[\tau_{\mathcal H}^2]$ denote the first and second moments of the exit time from $\mathcal H$. If $a>b$, then the first moment of the exit time from $\mathcal H$ satisfies
\[
u_1(x, y)=\frac{a^2b^2}{2(a^2-b^2)}\left(\frac{x^2}{a^2}-\frac{y^2}{b^2}-1 \right).
\]
If $a>(1+\sqrt{2})b$ then the second moment of the exit time from $\mathcal H$ satisfies
\[
u_2(x, y)=\left(Ax^2+By^2+C\right)u_1(x, y),
\]
where the constants $A, B,$ and $C$ satisfy
\begin{align*}
A&=\frac{b^2(5a^2-b^2)}{6(a^4-6a^2b^2+b^4)},\\
B&=-\frac{a^2(a^2-5b^2)}{6(a^4-6a^2b^2+b^4)},\\
C&=\frac{a^2b^2(5b^2-a^2)(5a^2-b^2)}{6(a^2-b^2)(a^4-6a^2b^2+b^4)}.
\end{align*}
\end{proposition}

\begin{proof}
Throughout the proof we denote
\begin{align*}
u_1^{\mathcal H}(x,y)&=\mathbb{E}^{(x,y)}[\tau_{\mathcal H}],\\
u_2^{\mathcal H}(x,y)&=\mathbb{E}^{(x,y)}[\tau_{\mathcal H}^2],
\end{align*}
and
\begin{align*}
v_1(x, y)&=\frac{a^2b^2}{2(a^2-b^2)}\left(\frac{x^2}{a^2}-\frac{y^2}{b^2}-1 \right),\\
v_2(x,y)&=\left(Ax^2+By^2+C\right)v_1(x, y),
\end{align*}
where the constants $A, B,$ and $C$ are as in the statement of the proposition. Then since $a>b$ it follows that $v_1(x,y)> 0$ in $\mathcal H$. We next show that $v_2(x,y)> 0$ in $\mathcal H$ under the assumption that $a>(1+\sqrt{2})b$. We start by considering a denominator term from the constants $A,B,C$, namely $a^4-6a^2b^2+b^4$. If we apply the quadratic formula with $a^2$ as our variable to the equation $a^4-6a^2b^2+b^4=0$, we see that
\[
a^2=\frac{6b^2 \pm \sqrt{36b^4-4b^4}}{2}=(3\pm 2\sqrt{2})b^2.
\]
It follows that
\begin{equation}\label{eq:factab}
a^4-6a^2b^2+b^4=\left(a^2-(3+ 2\sqrt{2})b^2\right)\left(a^2-(3- 2\sqrt{2})b^2\right).
\end{equation}
Since $3+2\sqrt{2}=(1+\sqrt{2})^2$, our assumption that $a>(1+\sqrt{2})b$ implies that $a^2>(3+2\sqrt{2})b^2$. And since $3+ 2\sqrt{2}>3- 2\sqrt{2}$ we also have $a^2>(3-2\sqrt{2})b^2$.
Thus from \eqref{eq:factab} we have that $a^4-6a^2b^2+b^4>0$. We now work with the formula for $v_2(x,y)$ as a whole. If $(x,y)\in \mathcal H$, then $x^2>\frac{a^2}{b^2}y^2+a^2$. Using this bound in the expression for $v_2$ we see
\begin{align*}
v_2(x,y)&\geq v_1(x,y)\left(A\left(\frac{a^2}{b^2}y^2+a^2\right)+By^2+C\right)\\
&=v_1(x,y)\left(\frac{2 a^2(a^4-b^4)}{3(a^2-b^2)(a^4-6a^2b^2+b^4)}y^2+\frac{2a^2b^4(5a^2-b^2)}{3(a^2-b^2)(a^4-6a^2b^2+b^4)}\right)
\end{align*}
which is positive since we have already argued that $a^4-6a^2b^2+b^4>0$ under the assumption that $a>(1+\sqrt{2})b$ and the remaining terms are all positive since $a>b$.

We next show that $u_1^{\mathcal H}(x,y)=v_1(x,y)$ in $\mathcal H$. Straightforward calculations show that $-\Delta v_1=1$ in $\mathcal H$ with $v_1=0$ on $\partial \mathcal H$. Denote
\[
\mathcal H_n=\left \{(x,y)\in \mathcal H: x<n\right \}
\]
and also
\[
u_1^{\mathcal H_n}(x,y)=\mathbb{E}^{(x,y)}[\tau_{\mathcal H_n}].
\]
In what follows we assume $n>a$. Since $\mathcal H_n$ is bounded, $u_1^{\mathcal H_n}$ satisfies
\[
\left\{
\begin{aligned}
-\Delta u_1^{\mathcal H_n} &= 1\ \textup{in }\mathcal H_n,\\
u_1^{\mathcal H_n} &= 0\ \textup{on }\partial \mathcal H_n,
\end{aligned}
\right.
\]
and since we've shown that $v_1\geq 0$ on $\mathcal H$, $v_1$ satisfies
\[
\left\{
\begin{aligned}
-\Delta v_1 &= 1\ \textup{in }\mathcal H_n,\\
v_1 &\geq  0\ \textup{on }\partial \mathcal H_n.
\end{aligned}
\right.
\]
It follows that $v_1-u_1^{\mathcal H_n}$ is harmonic in $\mathcal H_n$ and nonnegative on $\partial \mathcal H_n$. By the minimum principle, we thus have 
\begin{equation}\label{ineq:trunhypu1}
v_1\geq u_1^{\mathcal H_n} \textup{ in } \mathcal H_n.
\end{equation}
Arguing as in the proof of Proposition \ref{prop:limtau}, we see
\[
\lim_{n\to \infty}\tau_{\mathcal H_n}=\tau_{\mathcal H}
\]
and applying the monotone convergence theorem, we see
\begin{equation}\label{ineq:hypu1v1}
v_1\geq u_1^{\mathcal H} \textup{ in }\mathcal H.
\end{equation}
Inequality \eqref{ineq:hypu1v1} implies that $u_1^{\mathcal H}$ is everywhere finite and so is one solution to the PDE problem
\[
\left\{
\begin{aligned}
-\Delta u_1^{\mathcal H} &= 1\ \textup{in }\mathcal H,\\
u_1^{\mathcal H} &= 0\ \textup{on }\partial \mathcal H.
\end{aligned}
\right.
\]
We conclude that $v_1-u_1^{\mathcal H}$ is harmonic in $\mathcal H$ and vanishes on $\partial \mathcal H$. The map 
\[
\varphi(z)=\frac{\sqrt{a^2+b^2}}{2} \left( \left( z + \sqrt{z^2 - 1} \right)^{\frac{2}{\pi}\arctan\left(\frac{b}{a}\right)} + \left( z + \sqrt{z^2 - 1} \right)^{-\frac{2}{\pi}\arctan\left(\frac{b}{a}\right)} \right)
\]
furnishes a conformal mapping between the right half-plane
\[
\mathbb{H}^+=\{(x,y)\in \mathbb{R}^2:x>0\}
\]
and the hyperbolic domain $\mathcal H$ and maps $\partial \mathbb{H}^+$ into $\partial \mathcal H$. Here the complex power functions are defined using their principal branches, with branch cuts taken along the negative real axis. Since harmonicity is preserved under conformal mapping, it follows that
\[
h_1(z)=v_1(\varphi(z))-u_1^{\mathcal H}(\varphi(z)),\ z\in \mathbb{H}^+
\]
is harmonic in $\mathbb{H}^+$ and vanishes on $\partial \mathbb{H}^+$. Moreover, by \eqref{ineq:hypu1v1} $h_1$ is nonnegative.  Since $u_1^{\mathcal H}\geq 0$, the formula for $v_1$ gives
\begin{equation}\label{eq:hbigobound1}
|h_1(z)|\leq |v_1(\varphi(z))|=\mathcal O\left(|z|^{\frac{4}{\pi}\arctan\left(\frac{b}{a}\right)}\right),\  z\in \mathbb{H}^+,\  |z|\to \infty.
\end{equation}
By Schwarz reflection, the odd extension of $h_1$ across the $y$-axis is harmonic in $\mathbb{R}^2$ and  satisfies the same bound in $\eqref{eq:hbigobound1}$ on $\mathbb{R}^2$. Our assumption that $a>b$ implies that $\frac{4}{\pi}\arctan\left(\frac{b}{a}\right)<1$ and so standard results about harmonic functions give that $h_1$ is constant (see exercise 7 on p.42 of \cite{A2}, for instance). But since $h_1=0$ on the $y-$axis, we see $h_1=0$ everywhere giving $u_1^{\mathcal H}(x,y)=v_1(x,y)$ in $\mathcal H$.

To show that $u_2^{\mathcal H}(x,y)=v_2(x,y)$ in $\mathcal H$ we proceed in a similar fashion. Defining
\[
u_2^{\mathcal H_n}(x,y)=\mathbb{E}^{(x,y)}[\tau_{\mathcal H_n}^2],
\]
we see that $u_2^{\mathcal H_n}$ and $v_2$ satisfy
\begin{equation*}
\left\{
\begin{aligned}
-\Delta u_2^{\mathcal H_n} &= 2u_1^{\mathcal H_n}\ \textup{in }\mathcal H_n,\\
u_2^{\mathcal H_n} &= 0\ \textup{on }\partial \mathcal H_n,
\end{aligned}
\right.
\qquad\qquad\qquad 
\left\{
\begin{aligned}
-\Delta v_2 &= 2v_1\ \textup{in }\mathcal H_n,\\
v_2 &\geq  0\ \textup{on }\partial \mathcal H_n.
\end{aligned}
\right.
\end{equation*}
Using \eqref{ineq:trunhypu1}, we see that
\[
-\Delta(v_2-u_2^{\mathcal H_n})=2(v_1-u_1^{\mathcal H_n})\geq 0
\]
and $v_2-u_2^{\mathcal H_n}\geq 0$ on $\partial \mathcal H_n$. By the minimum principle for superharmonic functions, we have that
\begin{equation}\label{ineq:trunhypu2}
v_2\geq u_2^{\mathcal H_n} \textup{ in } \mathcal H_n.
\end{equation}
Letting $n\to \infty$ and using the monotone convergence theorem, we see
\begin{equation}\label{ineq:hypu2v2}
v_2\geq u_2^{\mathcal H} \textup{ in }\mathcal H.
\end{equation}
Inequality \eqref{ineq:hypu2v2} above implies that $u_2^{\mathcal H}$ is everywhere finite, and so $u_2^{\mathcal H}$ is one solution to the PDE problem
\[
\left\{
\begin{aligned}
-\Delta u_2^{\mathcal H} &= 2u_1^{\mathcal H}\ \textup{in }\mathcal H,\\
u_2^{\mathcal H} &= 0\ \textup{on }\partial \mathcal H.
\end{aligned}
\right.
\]
Since we've already shown that $u_1^{\mathcal H}=v_1$ in $\mathcal H$, we conclude that $v_2-u_2^{\mathcal H}$ is harmonic in $\mathcal H$ and vanishes on $\partial \mathcal H$. Since harmonicity is preserved under conformal mapping,
\[
h_2(z)=v_2(\varphi(z))-u_2^{\mathcal H}(\varphi(z)),\ z\in \mathbb{H}^+
\]
is harmonic in $\mathbb{H}^+$ and vanishes on $\partial \mathbb{H}^+$. By inequality \ref{ineq:hypu2v2}, $h_2$ is nonnegative. Moreover, since $u_2^{\mathcal H}\geq 0$, the formula for $v_2$ gives
\begin{equation}\label{eq:hbigobound2}
|h_2(z)|\leq |v_2(\varphi(z))|=\mathcal O\left(|z|^{\frac{8}{\pi}\arctan\left(\frac{b}{a}\right)}\right),\  z\in \mathbb{H}^+,\  |z|\to \infty.
\end{equation}
By Schwarz reflection, the odd extension of $h_2$ across the $y$-axis is harmonic in $\mathbb{R}^2$ and  satisfies the same bound in $\eqref{eq:hbigobound2}$ on $\mathbb{R}^2$. Our assumption that $a>(1+\sqrt{2})b$ implies $\frac{8}{\pi}\arctan\left(\frac{b}{a}\right)<1$ and as before $h_2$ is constant. But since $h_2=0$ on the $y$-axis, we see $h_2=0$ everywhere giving $u_2^{\mathcal H}(x,y)=v_2(x,y)$ in $\mathcal H$.

\end{proof}

\begin{theorem}\label{thm:mainreshyp}
Let $\mathcal H$ be our hyperbolic domain and assume that $a>(1+\sqrt{2})b$. Let $u_1(x,y)=\mathbb{E}^{(x,y)}[\tau_{\mathcal H}]$ and $u_2(x,y)=\mathbb{E}^{(x,y)}[\tau_{\mathcal H}^2]$  be the first two moments of the exit time from $\mathcal H$. Given positive real numbers $c_1$ and $c_2$, if the level curves $u_1(x,y)=c_1$ and $u_2(x,y)=c_2$ intersect, then their intersection is either two points or a single point determined by $c_1$ and $c_2$.
\end{theorem}

\begin{proof}
Assume that the level curves $u_1(x,y) = c_1$ and $u_2(x,y) = c_2$ intersect for some positive constants $c_1$ and $c_2$. Using Proposition \ref{prop:momformshyp}, we solve the equation $u_1(x,y) = c_1$ for $x^2$ and $y^2$ to obtain:
\begin{align*}
x^2&=a^2\left(\frac{y^2}{b^2}+1+c_1\frac{2(a^2-b^2)}{a^2 b^2}\right),\\
y^2&=b^2 \left(\frac{x^2}{a^2}-1-c_1\frac{2(a^2-b^2)}{a^2 b^2}\right).
\end{align*}
Each of the equations above can be substituted back into our equation $u_2(x,y) = c_2$ giving
\begin{equation}\label{eq:hyperbolac2c1}
    \frac{c_2}{c_1} = Dx^2 + E = Fy^2 + G,
\end{equation}
for some constants $D,E,F,G$ each depending on $a$ and $b$:
\begin{align*}
    D &= \frac{2b^2(a^4-b^4)}{3(a^2-b^2)(a^4-6a^2b^2+b^4)}, \\
   E &= \frac{(a^2-5 b^2)(c_1(a^2-b^2)^2-2a^4b^2)}{3(a^2-b^2)(a^4-6a^2b^2+b^4)}, \\
   F &= \frac{2a^2(a^4-b^4)}{3(a^2-b^2)(a^4-6a^2b^2+b^4)}  , \\
   G &= \frac{(5a^2-b^2)(2a^2b^4+c_1(a^2-b^2)^2)}{3(a^2-b^2)(a^4-6a^2b^2+b^4)}.
\end{align*}
Since $a>(1+\sqrt{2})b$, the constants $D$ and $F$ are nonzero. In particular, we can solve \eqref{eq:hyperbolac2c1} explicitly for $x^2$ and $y^2$ giving
\begin{align*}
x^2&=\frac{1}{D}\left(\frac{c_2}{c_1}-E\right),\\
y^2&=\frac{1}{F}\left(\frac{c_2}{c_1}-G\right).
\end{align*}
Since we assumed that the level curves $u_1(x,y) = c_1$ and $u_2(x,y) = c_2$ intersect, the expression $\frac{1}{D}\left(\frac{c_2}{c_1}-E\right)$ must be greater than $a^2$ and $\frac{1}{F}\left(\frac{c_2}{c_1}-G\right)$ must be nonnegative. If $y^2=\frac{1}{F}\left(\frac{c_2}{c_1}-G\right)=0$, then solving $x^2=\frac{1}{D}\left(\frac{c_2}{c_1}-E\right)$ for $x$ and choosing the positive solution, we have reduced $(x,y)$ to one possible location along the $x$-axis. Thus the level curves intersect at a single point. If $y^2=\frac{1}{F}\left(\frac{c_2}{c_1}-G\right)>0$, then solving for $y$, and again solving $x^2=\frac{1}{D}\left(\frac{c_2}{c_1}-E\right)$ for $x$ and choosing the positive solution, we have reduced $(x,y)$ to two possible locations, reflection symmetric across the $x$-axis. Here the level curves intersect at two points.
\end{proof}

As with our previous settings, we have the following corollary to Theorem \ref{thm:mainreshyp} on location detection.

\begin{corollary}\label{cor:locdethyp}
Consider a hyperbolic domain $\mathcal H$ with $a>(1+\sqrt{2})b$. Given values of the first two moments of the exit time at some point unknown point $(x,y)$ in $\mathcal H$, one can detect the location $(x,y)$ up to symmetry. Equivalently, given the mean exit time and variance of the exit time at some unknown point $(x,y)$ in $\mathcal H$, one can detect the location $(x,y)$ up to symmetry.
\end{corollary}

\section{Wedge Domains}\label{sect:wedge}

Throughout this section we let $\mathcal W$ denote the unbounded wedge domain
\[
\mathcal W=\left \{z=re^{i\theta}\in \mathbb{R}^2:0<r<\infty,\ -\frac{\alpha}{2}<\theta<\frac{\alpha}{2}\right \},
\]
where $0<\alpha<\pi$ is the opening angle. See Figure \ref{fig:wedge} below.

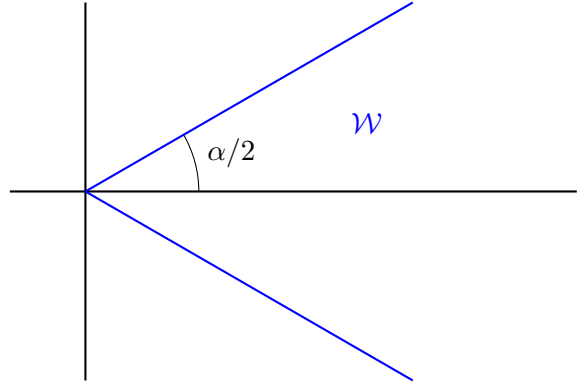
\begin{figure}[H]
\begin{center}
\begin{tikzpicture}

    \def\radius{5}
    \def\halfalpha{30} 

    \draw[-, thick] (-1,0) -- (\radius+1.5,0) node[right, text=black]{};
    \draw[-, thick] (0,-2.5) -- (0,2.5) node[above, text=black]{};

    \draw[thick, blue] (0,0) -- (\halfalpha:\radius) node[right] {};
    \draw[thick, blue] (0,0) -- (-\halfalpha:\radius) node[right] {};

    \draw (1.5,0) arc (0:\halfalpha:1.5);
    \node at (\halfalpha/2:2) {$\alpha/2$};

    \node[right, blue] at (\halfalpha/2:3.5) {$\mathcal{W}$};

\end{tikzpicture}
\end{center}
\caption{A picture of the wedge domain $\mathcal W$ and its boundary.}
\label{fig:wedge}
\end{figure}

Our first result gives explicit formulas in polar variables for the first two exit time moments in a wedge.

\begin{proposition}\label{prop:pepformsswedge}
Let $\mathcal W$ denote the wedge domain above and let $u_1(r,\theta)=\mathbb{E}^{(r,\theta)}[\tau_{\mathcal W}]$ and $u_2(r,\theta)=\mathbb{E}^{(r,\theta)}[\tau_{\mathcal W}^2]$ denote the first and second moments of the exit time from $\mathcal W$. If $\alpha<\frac{\pi}{2}$, then the first moment of the exit time from $\mathcal W$ satisfies
\[
u_1(r, \theta)=\frac{r^2}{4}\left( \frac{\cos (2\theta) }{\cos (\alpha) } - 1 \right).
\]
If $\alpha<\frac{\pi}{4}$, then the second moment of the exit time from $\mathcal W$ satisfies
\[
u_2(r,\theta)=u_1(r,\theta) \cdot \left( \frac{r^2}{24 \cos(2\alpha)} \left( 2\cos(\alpha)\cos(2\theta) + 1 - 3\cos(2\alpha) \right) \right).
\]
\end{proposition}

\begin{proof}
Throughout the proof we denote
\begin{align*}
u_1^{\mathcal W}(r,\theta)&=\mathbb{E}^{(r,\theta)}[\tau_{\mathcal W}],\\
u_2^{\mathcal W}(r,\theta)&=\mathbb{E}^{(r,\theta)}[\tau_{\mathcal W}^2],
\end{align*}
and
\begin{align*}
v_1(r, \theta)&=\frac{r^2}{4}\left( \frac{\cos (2\theta) }{\cos (\alpha) } - 1 \right),\\
v_2(r,\theta)&=v_1(r,\theta) \cdot \left( \frac{r^2}{24 \cos(2\alpha)} \left( 2\cos(\alpha)\cos(2\theta) + 1 - 3\cos(2\alpha) \right) \right). 
\end{align*}
We first show that when $0<\alpha<\frac{\pi}{2}$, one has $u_1^{\mathcal W}(r,\theta)=v_1(r,\theta)$ in $\mathcal W$. Writing $\Delta$ in polar coordinates, straightforward calculations show that $-\Delta v_1=1$ in $\mathcal W$ with $v_1=0$ on $\partial \mathcal W$ and $-\Delta v_2=2v_1$ in $\mathcal W$ with $v_2=0$ on $\partial \mathcal W$. It is also easy to see that $v_1$ and $v_2$ are both positive (under their respective restrictions on $\alpha$) in $\mathcal W$ and vanish on $\partial \mathcal W$. Denote
\[
\mathcal W_n=\left \{z=re^{i\theta}\in \mathbb{R}^2:0<r<n,\ -\frac{\alpha}{2}<\theta<\frac{\alpha}{2}\right \}
\]
and also
\[
u_1^{\mathcal W_n}(r,\theta)=\mathbb{E}^{(r,\theta)}[\tau_{\mathcal W_n}].
\]
Then $u_1^{\mathcal W_n}$ and $v_1$ satisfy
\begin{equation*}
\left\{
\begin{aligned}
-\Delta u_1^{\mathcal W_n} &= 1\ \textup{in }\mathcal W_n,\\
u_1^{\mathcal W_n} &= 0\ \textup{on }\partial \mathcal W_n,
\end{aligned}
\right.
\qquad\qquad\qquad 
\left\{
\begin{aligned}
-\Delta v_1 &= 1\ \textup{in }\mathcal W_n,\\
v_1 &\geq  0\ \textup{on }\partial \mathcal W_n.
\end{aligned}
\right.
\end{equation*}
It follows that $v_1-u_1^{\mathcal W_n}$ is harmonic in $\mathcal W_n$ and nonnegative on $\partial \mathcal W_n$. By the minimum principle, we thus have 
\begin{equation}\label{ineq:trunwedu1}
v_1\geq u_1^{\mathcal W_n} \textup{ in } \mathcal W_n.
\end{equation}
Letting $n\to \infty$ and using the monotone convergence theorem, we see
\[
v_1\geq u_1^{\mathcal W} \textup{ in }\mathcal W.
\]
The displayed inequality above implies that $u_1^{\mathcal W}$ is everywhere finite and so is one solution to the PDE problem
\[
\left\{
\begin{aligned}
-\Delta u_1^{\mathcal W} &= 1\ \textup{in }\mathcal W,\\
u_1^{\mathcal W} &= 0\ \textup{on }\partial \mathcal W.
\end{aligned}
\right.
\]
We conclude that $v_1-u_1^{\mathcal W}$ is harmonic in $\mathcal W$ and vanishes on $\partial \mathcal W$. The map $z\to z^{\frac{\alpha}{\pi}}$ furnishes a conformal mapping between the right half-plane
\[
\mathbb{H}^+=\{(x,y)\in \mathbb{R}^2:x>0\}
\]
and the wedge $\mathcal W$ and takes $\partial \mathbb{H}^+$ into $\partial \mathcal W$. Again we take our branch cut along the negative real axis. Since harmonicity is preserved under conformal mapping, it follows that
\[
h_1(z)=v_1(z^{\frac{\alpha}{\pi}})-u_1^{\mathcal W}(z^{\frac{\alpha}{\pi}}),\ z\in \mathbb{H}^+
\]
is harmonic in $\mathbb{H}^+$ and vanishes on $\partial \mathbb{H}^+$. Moreover, since $u_1^{\mathcal W}\geq 0$, the formula for $v_1$ gives
\begin{equation}\label{eq:hbigobound1wedge}
|h_1(z)|\leq |v_1(z^{\frac{\alpha}{\pi}})|=\mathcal O(|z|^{\frac{2\alpha}{\pi}}),\  z\in \mathbb{H}^+,\  |z|\to \infty.
\end{equation}
By Schwarz reflection, the odd extension of $h_1$ across the $y$-axis is harmonic in $\mathbb{R}^2$ and  satisfies the same bound in $\eqref{eq:hbigobound1wedge}$ on $\mathbb{R}^2$. Our assumption that $\alpha < \frac{\pi}{2}$ implies $\frac{2\alpha}{\pi}<1$ and so as before $h_1$ is constant. But since $h_1=0$ on the $y$-axis, we see $h_1=0$ everywhere giving $u_1^{\mathcal W}(r,\theta)=v_1(r,\theta)$ in $\mathcal W$.

To show that $u_2^{\mathcal W}(r,\theta)=v_2(r,\theta)$ in $\mathcal W$ define
\[
u_2^{\mathcal W_n}(r,\theta)=\mathbb{E}^{(r,\theta)}[\tau_{\mathcal W_n}^2].
\]
We see that $u_2^{\mathcal W_n}$ and $v_2$ satisfy
\begin{equation*}
\left\{
\begin{aligned}
-\Delta u_2^{\mathcal W_n} &= 2u_1^{\mathcal W_n}\ \textup{in }\mathcal W_n,\\
u_2^{\mathcal W_n} &= 0\ \textup{on }\partial \mathcal W_n,
\end{aligned}
\right.
\qquad\qquad\qquad 
\left\{
\begin{aligned}
-\Delta v_2 &= 2v_1\ \textup{in }\mathcal W_n,\\
v_2 &\geq  0\ \textup{on }\partial \mathcal W_n.
\end{aligned}
\right.
\end{equation*}
By \eqref{ineq:trunwedu1}, we thus have that
\[
-\Delta(v_2-u_2^{\mathcal W_n})=2(v_1-u_1^{\mathcal W_n})\geq 0
\]
and $v_2-u_2^{\mathcal W_n}\geq 0$ on $\partial \mathcal W_n$. By the minimum principle for superharmonic functions, we have that
\begin{equation}\label{ineq:trunwedu2}
v_2\geq u_2^{\mathcal W_n} \textup{ in } \mathcal W_n.
\end{equation}
Letting $n\to \infty$ and using the monotone convergence theorem, we see
\[
v_2\geq u_2^{\mathcal W} \textup{ in }\mathcal W.
\]
The displayed inequality above implies that $u_2^{\mathcal W}$ is everywhere finite, and so $u_2^{\mathcal W}$ is one solution to the PDE problem
\[
\left\{
\begin{aligned}
-\Delta u_2^{\mathcal W} &= 2u_1^{\mathcal W}\ \textup{in }\mathcal W,\\
u_2^{\mathcal W} &= 0\ \textup{on }\partial \mathcal W.
\end{aligned}
\right.
\]
We conclude that $v_2-u_2^{\mathcal W}$ is harmonic in $\mathcal W$ and vanishes on $\partial \mathcal W$. Since harmonicity is preserved under conformal mapping,
\[
h_2(z)=v_2(z^{\frac{\alpha}{\pi}})-u_2^{\mathcal W}(z^{\frac{\alpha}{\pi}}),\ z\in \mathbb{H}^+
\]
is harmonic in $\mathbb{H}^+$ and vanishes on $\partial \mathbb{H}^+$. Moreover, since $u_2^{\mathcal W}\geq 0$, the formula for $v_2$ gives
\begin{equation}\label{eq:hbigobound2wedge}
|h_2(z)|\leq |v_2(z^{\frac{\alpha}{\pi}})|=\mathcal O(|z|^{\frac{4\alpha}{\pi}}),\  z\in \mathbb{H}^+,\  |z|\to \infty.
\end{equation}
By Schwarz reflection, the odd extension of $h_2$ across the $y$-axis is harmonic in $\mathbb{R}^2$ and  satisfies the same bound in $\eqref{eq:hbigobound2wedge}$ on $\mathbb{R}^2$. Our assumption that $\alpha < \frac{\pi}{4}$ implies $\frac{4\alpha}{\pi}<1$ and as before $h_2$ is constant. But since $h_2=0$ on the $y-$axis, we see $h_2=0$ everywhere giving $u_2^{\mathcal W}(r,\theta)=v_2(r,\theta)$ in $\mathcal W$.
\end{proof}

\begin{theorem}\label{thm:wedge}
Let $\mathcal W$ be our wedge domain and let $u_1(r,\theta)=\mathbb{E}^{(r,\theta)}[\tau_{\mathcal W}]$ and $u_2(r,\theta)=\mathbb{E}^{(r,\theta)}[\tau_{\mathcal W}^2]$  be the first two moments of the exit time from $\mathcal W$. Assume $\alpha<\frac{\pi}{4}$. Given positive numbers $c_1$ and $c_2$, if the level curves $u_1(r,\theta)=c_1$ and $u_2(r,\theta)=c_2$ intersect, then their intersection is either two points or a single point determined by $c_1$ and $c_2$.
\end{theorem}

\begin{proof}
Assuming that the level curves $u_1(r,\theta)=c_1$ and $u_2(r,\theta)=c_2$ intersect, we use the formulas from Proposition \ref{prop:pepformsswedge}. Solving $u_1(r,\theta)=c_1$ for $\cos(2\theta)$ gives
\begin{equation}\label{eq:cos2th}
\cos(2\theta)=\left(\frac{4c_1}{r^2}+1\right)\cos(\alpha).
\end{equation}
Plugging \eqref{eq:cos2th} into $u_2(r,\theta)=c_2$ and solving for $r^2$ gives
\begin{equation}\label{eq:rsqwedge}
r^2=\frac{\frac{6c_2}{c_1}\cos(2\alpha) - 2c_1\cos^2(\alpha)}{\sin^2(\alpha)}.
\end{equation}
By assumption, \eqref{eq:rsqwedge} has one positive solution in $r$. Plugging this $r$-value into \eqref{eq:cos2th}, our assumption on the level curves gives the existence of at least one solution to \eqref{eq:cos2th} in $\theta$ for some $-\frac{\alpha}{2}<\theta<\frac{\alpha}{2}$. If $\left(\frac{4c_1}{r^2}+1\right)\cos(\alpha)=1$, then $\theta=0$ and the level curves $u_1(r,\theta)=c_1$ and $u_2(r,\theta)=c_2$ intersect at a single point in $\mathcal W$ on the $x$-axis. If $\left(\frac{4c_1}{r^2}+1\right)\cos(\alpha)<1$, then since $\cos(2\theta)$ is strictly decreasing on $\left(0, \frac{\pi}{4}\right)$ and $\alpha<\frac{\pi}{4}$, it follows that \eqref{eq:cos2th} has precisely two solutions in $\left(-\frac{\alpha}{2},\frac{\alpha}{2}\right)$, one positive and one negative. We deduce that the level curves $u_1(r,\theta)=c_1$ and $u_2(r,\theta)=c_2$ intersect at two points in $\mathcal W$, symmetric across the $x$-axis.
\end{proof}

As before, we have the following corollary to Theorem \ref{thm:wedge} on location detection.

\begin{corollary}\label{cor:locdetwedge}
Consider a wedge domain $\mathcal W$ with $\alpha<\frac{\pi}{4}$. Given values of the first two moments of the exit time at some unknown point $(x,y)$ in $\mathcal W$, one can detect the location $(x,y)$ up to symmetry. Equivalently, given the mean exit time and variance of exit time at some unknown point $(x,y)$ in $\mathcal W$, one can detect the location $(x,y)$ up to symmetry.
\end{corollary}

\section{Equilateral Triangular Domains}\label{sect:tri}

Throughout this section we let $\mathcal T$ denote the equilateral triangle interior to the three lines
\begin{align*}
y&=0,\\
y&=\sqrt{3}x + b, \\
y&=-\sqrt{3}x + b,
\end{align*}
for some positive real number $b$. The triangle $\mathcal T$ has three median lines or lines of symmetry prescribed by $x = 0$,  $y = \frac{x\sqrt3}{3} + \frac{b}{3}, $ and $ y = -\frac{x\sqrt3}{3} + \frac{b}{3}.$ See Figure \ref{fig:tri} below.

\begin{center}
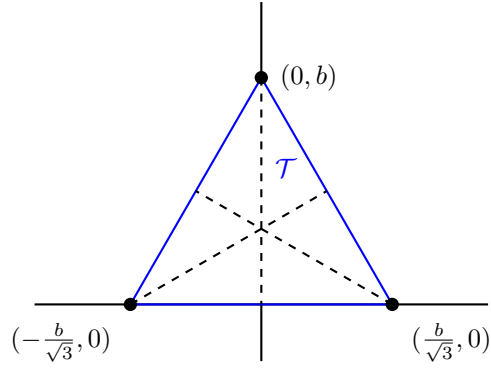
\begin{figure}[H]
    \centering

\begin{tikzpicture}[scale=1, font=\small]

    \def\b{3} 
    \pgfmathsetmacro{\xintercept}{\b/sqrt(3)}
    \pgfmathsetmacro{\centroidY}{\b/3}

    \coordinate (Top) at (0, \b);
    \coordinate (Left) at (-\xintercept, 0);
    \coordinate (Right) at (\xintercept, 0);
    \coordinate (Origin) at (0,0);
    
    \coordinate (MidLeft) at (-\xintercept/2, \b/2);
    \coordinate (MidRight) at (\xintercept/2, \b/2);
    \coordinate (MidBase) at (0,0);


        ({-\xintercept - 0.5}, {-sqrt(3)*0.5}) -- ({0.5}, {\b + sqrt(3)*0.5});
        ({\xintercept + 0.5}, {-sqrt(3)*0.5}) -- ({-0.5}, {\b + sqrt(3)*0.5});
    \draw[thick, black] (-3, 0) -- (3, 0);
    \draw[thick, black] (0, 0) -- (0, -.75);

    \draw[thick, blue] (Left) -- (Top) -- (Right) -- cycle;

    \draw[black, thick] (0, 4) -- (Top);               
    \draw[dashed, black, thick] (Top) -- (MidBase);    
    \draw[dashed, black, thick] (Left) -- (MidRight);
    \draw[dashed, black, thick] (Right) -- (MidLeft);

    \fill (0,\b) circle (2.5pt) ;
    \fill (Left) circle (2.5pt) ;
    \fill (Right) circle (2.5pt) ;
    
    \node[right = 0.3em] at (Top) {$(0, b)$};
    \node[below left = 0.3em] at (Left) {$(-\frac{b}{\sqrt3},0)$};
    \node[below right = 0.3em] at (Right)  {$(\frac{b}{\sqrt3},0)$};
    \node[left, blue] at (.6, 1.8) {$\mathcal{T}$};

\end{tikzpicture}
\caption{The triangular region $\mathcal T$ together with its dashed median lines.}
\label{fig:tri}
\end{figure}
\end{center}
The main result of this section is the following.
\begin{theorem}\label{th:eqtr}
Let $\mathcal T$ be the equilateral triangular domain above. Let $u_1(x,y)=\mathbb{E}^{(x,y)}[\tau_{\mathcal T}]$ and $u_2(x,y)=\mathbb{E}^{(x,y)}[\tau_{\mathcal T}^2]$  be the first two moments of the exit time from $\mathcal T$. Then
\begin{align*}
u_1(x,y)&= \frac{1}{4b}y(y-\sqrt{3}x-b)(y+\sqrt{3}x-b)=\frac{y((y-b)^2 -3 x^2 ) }{4 b},\\
u_2(x,y)&= u_1(x,y)\left(\frac{b^2 + 2 b y - 3 (x^2 + y^2)}{24}\right).
\end{align*}
Moreover, given positive real numbers $c_1,c_2$, if the level curves $u_1(x,y)=c_1$ and $u_2(x,y)=c_2$ intersect, then they intersect in either one, three, or six points depending on $c_1$ and $c_2$.
\end{theorem}

\begin{proof}
If $u_1(x,y)$ and $u_2(x,y)$ denote the polynomials in the theorem statement, then it is straightforward to check that they satisfy
\begin{equation*}
\left\{
\begin{aligned}
-\Delta u_1 &= 1\ \textup{in }\mathcal T,\\
u_1 &= 0\ \textup{on }\partial \mathcal T,
\end{aligned}
\right.
\qquad\qquad\qquad 
\left\{
\begin{aligned}
-\Delta u_2 &= 2u_1\ \textup{in }\mathcal T,\\
u_2 &= 0\ \textup{on }\partial \mathcal T.
\end{aligned}
\right.
\end{equation*}
Since $\mathcal T$ is bounded, $u_1(x,y)=\mathbb{E}^{(x,y)}[\tau_{\mathcal T}]$ and $u_2(x,y)=\mathbb{E}^{(x,y)}[\tau_{\mathcal T}^2]$ have the claimed formulas for the exit time moments.

Solving $u_1(x,y)=c_1$ for $x^2$ gives
\begin{equation}\label{eq:trixse}
x^2=\frac{1}{3}\left((y-b)^2-\frac{4bc_1}{y}\right).
\end{equation}
Using this expression for $x^2$ together with the equation $u_1(x,y)=c_1$, we rewrite the equation $u_2(x,y)=c_2$ as
\begin{equation}\label{eq:triycubic}
y^3 - by^2 + \frac{6c_2}{c_1}y - bc_1 = 0.
\end{equation}
Note that there are at most three distinct real solutions $y$ to \eqref{eq:triycubic} and that for each such solution $y$, there are at most two distinct real solutions $x$ to \eqref{eq:trixse}. Therefore, there are a maximum of six solutions $(x,y)$ that satisfy $u_1(x,y) = c_1$ and $u_2(x,y)=c_2$.

To better pin down the number of solutions, we note that $u_1(x,y)$ and $u_2(x,y)$ are invariant under the maps
\begin{align*}
(x,y)&\mapsto (-x,y),\\
(x,y)&\mapsto \left(-\frac{x}{2} - \frac{y\sqrt{3}}{2} + \frac{b\sqrt3}{6}, 
\frac{\sqrt{3}}{2}x - \frac{y}{2} + \frac{b}{2}\right).
\end{align*}
That is, the functions $u_1$ and $u_2$ are invariant under reflection across the $y$-axis and under $120^{\circ}$ counter-clockwise rotation about $\left(0, \frac{b}{3}\right)$. It follows that if $x$ and $y$ solve $\eqref{eq:trixse}$ and $\eqref{eq:triycubic}$, then other solutions may be obtained through (potentially repeated) rotation and reflection of the point $(x,y)$.

Since the level curves $u_1=c_1$ and $u_2=c_2$ intersect, the equations \eqref{eq:trixse} and \eqref{eq:triycubic} each have at least one solution $(x^\ast,y^\ast)\in \mathcal T$. In what follows we fix this known solution and consider what happens based on its location within $\mathcal T$.  See Figure \ref{fig:threeposs} below.

\begin{figure}[H]
\noindent
\begin{minipage}[t]{0.3\textwidth}
    \centering \textbf{Possibility $1$} \\
    \begin{center}
    \begin{tikzpicture}[scale=1, font=\small]

    \def\b{3} 
    \pgfmathsetmacro{\xintercept}{\b/sqrt(3)}
    \pgfmathsetmacro{\centroidY}{\b/3}

    \coordinate (Top) at (0, \b);
    \coordinate (Left) at (-\xintercept, 0);
    \coordinate (Right) at (\xintercept, 0);
    \coordinate (Origin) at (0,0);
    \coordinate (Centroid) at (0, \centroidY);
    
    \coordinate (MidLeft) at (-\xintercept/2, \b/2);
    \coordinate (MidRight) at (\xintercept/2, \b/2);
    \coordinate (MidBase) at (0,0);

    \draw[thick, blue] (Left) -- (Top) -- (Right) -- cycle;

    \draw[dashed, black, thick] (Top) -- (MidBase);    
    \draw[dashed, black, thick] (Left) -- (MidRight);
    \draw[dashed, black, thick] (Right) -- (MidLeft);

    \coordinate (P1) at (.3, 2.1);
    \node[right=5pt] at (.3, 2.1) {$(x^{\ast},y^{\ast})$};
    
    \coordinate (P1_ref) at (-.3, 2.1);

    \coordinate (P2) at ([rotate around={120:(Centroid)}]P1);
    \coordinate (P3) at ([rotate around={240:(Centroid)}]P1);
    
    \coordinate (P4) at ([rotate around={120:(Centroid)}]P1_ref);
    \coordinate (P5) at ([rotate around={240:(Centroid)}]P1_ref);

    \fill[black] (P1) circle (1.5pt);
    \fill[black] (P2) circle (1.5pt);
    \fill[black] (P3) circle (1.5pt);
    \fill[black] (P1_ref) circle (1.5pt);
    \fill[black] (P4) circle (1.5pt);
    \fill[black] (P5) circle (1.5pt);

\end{tikzpicture}
\end{center}
\end{minipage}
\hfill
\begin{minipage}[t]{0.3\textwidth}
    \centering \textbf{Possibility $2$} \\
      \begin{center}
    \begin{tikzpicture}[scale=1, font=\small]

    \def\b{3} 
    \pgfmathsetmacro{\xintercept}{\b/sqrt(3)}
    \pgfmathsetmacro{\centroidY}{\b/3}

    \coordinate (Top) at (0, \b);
    \coordinate (Left) at (-\xintercept, 0);
    \coordinate (Right) at (\xintercept, 0);
    \coordinate (Origin) at (0,0);
    \coordinate (Centroid) at (0, \centroidY);
    
    \coordinate (MidLeft) at (-\xintercept/2, \b/2);
    \coordinate (MidRight) at (\xintercept/2, \b/2);
    \coordinate (MidBase) at (0,0);

    \draw[thick, blue] (Left) -- (Top) -- (Right) -- cycle;

    \draw[dashed, black, thick] (Top) -- (0,\b/3 - 0.5);
    \draw[dashed, black, thick] (0,0) -- (0,0.1);    
    \draw[dashed, black, thick] (Left) -- (MidRight);
    \draw[dashed, black, thick] (Right) -- (MidLeft);

    \coordinate (P1) at (0, \b/3);
    \node at (0, \b/8+.02) {$(x^{\ast},y^{\ast})$};
    
    \coordinate (P2) at ([rotate around={120:(Centroid)}]P1);
    \coordinate (P3) at ([rotate around={240:(Centroid)}]P1);

    \fill[black] (P1) circle (1.5pt);
    \fill[black] (P2) circle (1.5pt);
    \fill[black] (P3) circle (1.5pt);

\end{tikzpicture}
\end{center}
    
\end{minipage}
\hfill
\begin{minipage}[t]{0.3\textwidth}
    \centering\textbf{Possibility $3$} \\
    \begin{center}
    \begin{tikzpicture}[scale=1, font=\small]

    \def\b{3} 
    \pgfmathsetmacro{\xintercept}{\b/sqrt(3)}
    \pgfmathsetmacro{\centroidY}{\b/3}

    \coordinate (Top) at (0, \b);
    \coordinate (Left) at (-\xintercept, 0);
    \coordinate (Right) at (\xintercept, 0);
    \coordinate (Origin) at (0,0);
    \coordinate (Centroid) at (0, \centroidY);
    
    \coordinate (MidLeft) at (-\xintercept/2, \b/2);
    \coordinate (MidRight) at (\xintercept/2, \b/2);
    \coordinate (MidBase) at (0,0);

    \draw[thick, blue] (Left) -- (Top) -- (Right) -- cycle;

    \draw[dashed, black, thick] (Top) -- (MidBase);    
    \draw[dashed, black, thick] (Left) -- (MidRight);
    \draw[dashed, black, thick] (Right) -- (MidLeft);

    \coordinate (P1) at (0, 2.5);
    \node[right=7pt,] at (0, 2.5) {$(x^{\ast},y^{\ast})$};
    
    \coordinate (P2) at ([rotate around={120:(Centroid)}]P1);
    \coordinate (P3) at ([rotate around={240:(Centroid)}]P1);

    \fill[black] (P1) circle (1.5pt);
    \fill[black] (P2) circle (1.5pt);
    \fill[black] (P3) circle (1.5pt);

\end{tikzpicture}
\end{center}
\end{minipage}
\caption{Three possible locations of the known solution $(x^{\ast},y^{\ast})$.}
\label{fig:threeposs}
\end{figure}
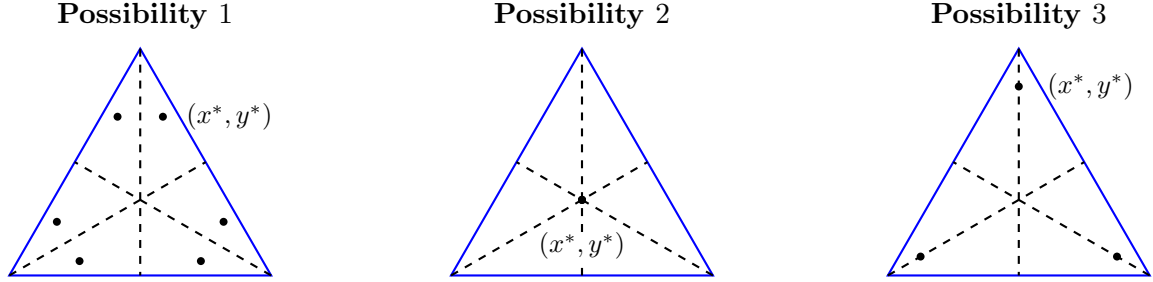

If $(x^{\ast},y^{\ast})$ does not fall along a median line of $\mathcal T$, then repeated rotation and reflection give six distinct points on the level curves $u_1=c_1$ and $u_2=c_2$. By our root counting, these points account for the full intersection of the level curves $u_1=c_1$ and $u_2=c_2$.

Next, suppose the centroid $(x^{\ast},y^{\ast}) = (0, \frac{b}{3})$ is the known solution. Using the explicit formulas for $u_1$ and $u_2$, we see that
$u_1(0, \frac{b}{3}) = \frac{b^2}{27}=c_1$ and $u_2(0, \frac{b}{3}) = \frac{b^4}{486}=c_2$. Plugging these values for $c_1$ and $c_2$ back into \eqref{eq:triycubic} gives
\begin{equation*}
     y^3  - by^2+ \frac{6c_2}{c_1}y  - bc_1 = \left(y- \frac{b}{3}\right)^3=0.
\end{equation*}
Consequently, $\frac{b}{3}$ is the only solution for $y$. Likewise, plugging into \eqref{eq:trixse} gives
\[
x^2=\frac{1}{3}\left((y-b)^2-\frac{4bc_1}{y}\right)=0
\]
and so $0$ is the only solution for $x$. We conclude that the level curves $u_1=c_1$ and $u_2=c_2$ intersect at a single point, namely $(0, \frac{b}{3})$.

Finally, suppose that our known solution $(x^{\ast},y^{\ast}) \ne (0, \frac{b}{3})$ lies along a median line of $\mathcal T$.  Consider the cubic polynomial
\[
P(y)=y^3 - by^2 + \frac{6c_2}{c_1}y - bc_1
\]
which comes from equation \eqref{eq:triycubic}. The discriminant of $P$, which we denote by $\textup{Disc}(P)$, satisfies
\begin{equation}\label{eq:disc}
    \operatorname{Disc}(P) = -\frac{4b^4c_1^4 + 27b^2c_1^5 - 108b^2c_1^3c_2 - 36b^2c_1c_2^2 + 864c_2^3}{c_1^3}.
\end{equation}
The value of the discriminant gives information about the roots of $P$. Using the level set equations $u_1=c_1$ and $u_2=c_2$, the discriminant can be rewritten as
\begin{equation}\label{eq:discPxy}
\textup{Disc}(P)=\frac{3x^2}{16}\left(3x^2-(b-3y)^2 \right)^2.
\end{equation}
Since $(x^{\ast}, y^{\ast})$ lies on a median line, it follows that $\operatorname{Disc}(P)=0$ and so $P$ has a repeated root. It follows that \eqref{eq:triycubic} has a maximum of two distinct real solutions in $y$, and by repeated rotation of $(x^\ast, y^\ast)$, \eqref{eq:triycubic} has exactly two solutions. The intersection of the level curves  $u_1=c_1$ and $u_2=c_2$ thus consists of three rotationally symmetric points about the centroid $(0,\frac{b}{3})$.
\end{proof}

As before, we have the following corollary to Theorem \ref{th:eqtr}.

\begin{corollary}\label{cor:locdettri}
Consider the equilateral triangular domain $\mathcal T$. Given values of the first two moments of the exit time at some unknown point $(x,y)$ in $\mathcal T$, one can detect the location $(x,y)$ up to symmetry. Equivalently, given the mean exit time and the variance of the exit time at some unknown point $(x,y)$ in $\mathcal T$, one can detect the location $(x,y)$ up to symmetry.
\end{corollary}

\begin{remark}
One might wonder whether the methods of this section generalize since other triangular domains can be realized as interiors of three lines. However, the equilateral triangle is unique in that it is the only triangular domain whose moments can be constructed as in Theorem \ref{th:eqtr}.
\end{remark}

\section{Rectangular Domains}\label{sect:rect}
Throughout this section we let $\mathcal R$ denote the rectangular domain
\[
\mathcal R = (0,a) \times (0,b)=\{(x,y)\in \mathbb{R}^2:0<x<a,\ 0<y<b\}.
\]
In order to consider non-square rectangles, we assume our rectangles are taller than they are wide. Consequently, we assume $b > a$. See Figure \ref{fig:rect} below.

\begin{figure}[H]
\begin{center}
\begin{tikzpicture}[scale=1.2]
    \def\a{3}
    \def\b{5}

    \draw[thick] (-0.5, 0) -- (\a + 1.5, 0) ;
    \draw[thick] (0, -0.5) -- (0, \b + 1);
    \draw[thick, blue] (0,0) -- (\a,0) -- (\a,\b) -- (0,\b) -- cycle;

    \fill (\a,0) circle (2.5pt) ;
    \fill (\a,\b) circle (2.5pt) ;
    \fill (0,\b) circle (2.5pt) ;
    \fill (0,0) circle (2.5pt) ;

    \node[below] at (\a, 0) {$(a,0)$};
    \node[left] at (0, \b) {$(0,b)$};
    \node[right] at (\a, \b) {$(a,b)$};
    \node[below left] at (0, 0) {$(0,0)$};
    
    \node[blue] at (\a/2, \b/2) {\Large $\mathcal R$};


\end{tikzpicture}
\end{center}
\caption{The rectangular domain $\mathcal R$ with width $a$ and height $b$.}
\label{fig:rect}
\end{figure}
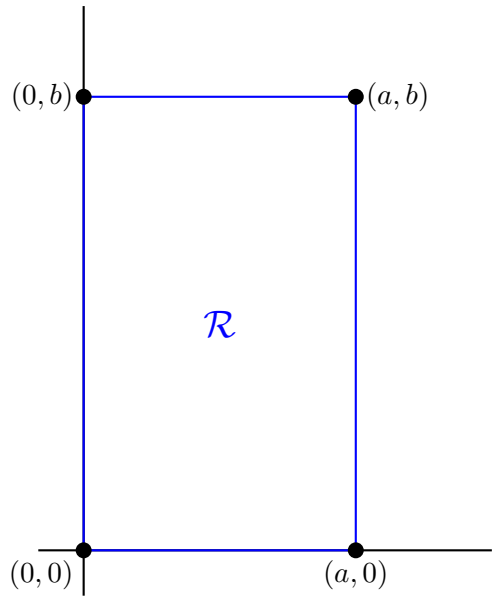

The eigenvalue problem for the Dirichlet Laplacian
 \[
       \left\{
        \begin{aligned}
            -\Delta \phi &= \lambda\phi \text{ in } \mathcal R, \\
            \phi &= 0  \text{ on } \partial \mathcal R, 
        \end{aligned}
        \right.
\]
admits a nontrivial solution for the eigenvalues
\begin{equation}\label{eq:evaluesrect}
\lambda_{m,n} = \left(\frac{m\pi}{a}\right)^2+\left(\frac{n\pi}{b}\right)^2
\end{equation}
and corresponding eigenfunctions
\begin{equation}\label{eq:efunctionsrect}
\phi_{m,n}(x,y)=\sin \left(\frac{m\pi x}{a}\right)\sin \left(\frac{n\pi y}{b}\right)
\end{equation}
for $m,n\geq 1$. These constants and functions can be found using the technique of separation of variables. Our next result gives series representations for the exit time moments in a rectangular domain.

\begin{proposition}\label{prop:momsrect}
Let $\mathcal R$ be a rectangle as above. Then the exit time moments $u_k(x,y)=\mathbb{E}^{(x,y)}[\tau_{\mathcal R}^k]$ satisfy
\[
u_k(x,y)=\sum_{\substack{m, n>0 \\ \text{$m, n$ odd}}}\frac{16k!}{\pi^2 mn\lambda_{m,n}^k}\sin \left(\frac{m\pi x}{a}\right)\sin \left(\frac{n\pi y}{b}\right),
\]
with $\lambda_{m,n}$ as in \eqref{eq:evaluesrect}.
\end{proposition}

\begin{proof}
We first consider the eigenexpansion of $u_1(x,y)=\mathbb{E}^{(x,y)}[\tau_{\mathcal R}]$:
\[
u_1(x,y)=\sum_{m,n>0}c_{m,n}\sin \left(\frac{m\pi x}{a}\right)\sin \left(\frac{n\pi y}{b}\right).
\]
Since $-\Delta u_1=1$ in $\mathcal R$, 
Applying $-\Delta$ to both sides of the equation above gives
\begin{equation}\label{eq:1eq}
1=\sum_{m,n>0}c_{m,n}\lambda_{m,n}\sin\left(\frac{m\pi x}{a}\right)\sin\left(\frac{n\pi y}{b}\right).
\end{equation} 
To solve for the constants $c_{m,n}$,  we multiply both sides in \eqref{eq:1eq} by $\sin\bigl(\frac{m'\pi x}{a}\bigr)\sin\bigl(\frac{n'\pi y}{b}\bigr)$ and integrate both sides over $\mathcal R$. The left integral is
\begin{align}
\int_0^a\int_0^b \sin\left(\frac{m'\pi x} {a}\right)\sin\left(\frac{n'\pi y}{b}\right)\, dy\, dx \nonumber
&=\frac{ab}{m'n'\pi^2}\left(\cos(m'\pi)-1\right)\left(\cos(n'\pi)-1\right) \nonumber\\ 
&=
\begin{cases}
\frac{4ab}{m'n'\pi^2} & \textup{if $m',n'$ are both odd},\\
0 & \textup{otherwise}.
\end{cases} \label{eq:leftsideint}
\end{align}
The right integral becomes
\begin{align}
&\int_0^a\int_0^b\left(\sum_{m,n>0}c_{m,n}\lambda_{m,n}\sin\left(\frac{m\pi x}{a}\right)\sin\left(\frac{n\pi y}{b}\right)\right)\sin\left(\frac{m'\pi x} {a}\right)\sin\left(\frac{n'\pi y}{b}\right)\,dy\,dx \nonumber \\
&=\sum_{m,n>0}c_{m,n}\lambda_{m,n}\int_0^a\int_0^b\sin\left(\frac{m\pi x}{a}\right)\sin\left(\frac{n\pi y}{b}\right)\sin\left(\frac{m'\pi x}{a}\right)\sin\left(\frac{n'\pi y}{b}\right)\,dy\,dx \nonumber \\
&=c_{m',n'}\lambda_{m',n'}\int_0^a \sin^2\left(\frac{m'\pi x}{a}\right)\,dx \int_0^b\sin^2\left(\frac{n'\pi y}{b}\right)\,dy \nonumber \\
&=c_{m',n'}\lambda_{m',n'}\frac{ab}{4}. \label{eq:rightsideint}
\end{align}
Equating \eqref{eq:leftsideint} and \eqref{eq:rightsideint} gives
\[
c_{m',n'}=
\begin{cases}
\frac{16}{\pi^2m'n'\lambda_{m',n'}} & \textup{if $m',n'$ are both odd},\\
0 & \textup{otherwise}.
\end{cases}
\]
Thus we have shown
\[
u_1(x,y)=\sum_{\substack{m,n>0 \\ \text{$m,n$ odd}}}\frac{16}{\pi^2 mn\lambda_{m,n}}\sin \left(\frac{m\pi x}{a}\right)\sin \left(\frac{n\pi y}{b}\right).
\]
To verify the formula for $u_k(x,y)$ in general, we can readily verify that
\[
u_k(x,y)=\sum_{\substack{m,n>0 \\ \text{$m,n$ odd}}}\frac{16k!}{\pi^2 mn\lambda_{m,n}^k}\sin \left(\frac{m\pi x}{a}\right)\sin \left(\frac{n\pi y}{b}\right),
\]
satisfies
\[
\left\{
\begin{aligned}
            -\Delta u_k &= k u_{k-1} \text{ in } \mathcal R, \\
            u_k &= 0 \text{ on } \partial \mathcal R. 
\end{aligned}
\right.
\]
Since the PDE problem above has a unique solution, we must have the correct formula for $u_k$.
\end{proof}

For our next result, we let $\mathcal R'$ denote the following subset of $\mathcal R$:
\[
\mathcal R'=\biggl(0,\frac{a}{2}\biggr]\times \biggl(0,\frac{b}{2}\biggr].
\]
Our proof that the exit time moments detect location in a rectangle hinges on the following proposition.
\begin{proposition}\label{prop:eigendisc}
    Assume $\mathcal R=(0,a)\times (0,b)$ is a rectangular domain with $a<b$. Let $\lambda_{m,n}$ and $\phi_{m,n}$ be as in \eqref{eq:evaluesrect} and \eqref{eq:efunctionsrect} and let $(x,y), (x', y') \in \mathcal R'$ be distinct. Then $\phi_{m,n}(x,y) \ne \phi_{m,n}(x', y')$ for at least one of $(m,n)=(1,1)$ or $(m,n)=(1,3)$.
\end{proposition}
\begin{proof}
Suppose that $\phi_{1,1}(x,y) = \phi_{1,1}(x', y')$ and $\phi_{1,3}(x,y) = \phi_{1,3}(x', y')$ for distinct points $(x,y), (x', y') \in \mathcal R'$. That is,
    \begin{align}
        \sin \left(\frac{\pi x}{a}\right)\sin\left(\frac{\pi y}{b}\right) &= \sin\left(\frac{\pi x'}{a}\right)\sin\left(\frac{\pi y'}{b}\right) \label{eq:1xxprime} \\
        \sin\left(\frac{\pi x}{a}\right)\sin\left(\frac{3\pi y}{b}\right) &= \sin\left(\frac{\pi x'}{a}\right)\sin\left(\frac{3\pi y'}{b}\right). \label{eq:2xxprime}
    \end{align}
The function $\sin\left(\frac{\pi z}{a}\right)$ is positive and one-to-one for $z\in\left(0,\frac{a}{2}\right]$ and similarly for $\sin\left(\frac{\pi z}{b}\right)$ on $\left(0,\frac{b}{2}\right]$. It follows that every term in \eqref{eq:1xxprime} is positive. Applying the trig identity $\sin(3\theta)=3\sin(\theta)-4\sin^3(\theta)$ and dividing \eqref{eq:2xxprime} by  \eqref{eq:1xxprime}, we see
\[
\frac{3\sin \left(\frac{\pi y}{b}\right) - 4\sin^3\left(\frac{\pi y}{b}\right)}{\sin\left(\frac{\pi y}{b}\right)}=\frac{3\sin \left(\frac{\pi y'}{b}\right) - 4\sin^3\left(\frac{\pi y'}{b}\right)}{\sin\left(\frac{\pi y'}{b}\right)}
\]
which implies $\sin^2\left(\frac{\pi y}{b}\right)=\sin^2\left(\frac{\pi y'}{b}\right)$. Since the sine terms are both positive, we get $\sin\left(\frac{\pi y}{b}\right)=\sin\left(\frac{\pi y'}{b}\right)$ which implies $y=y'$ by injectivity on $\left(0,\frac{b}{2}\right]$. Now \eqref{eq:1xxprime} gives $\sin\left(\frac{\pi x}{a}\right)=\sin\left(\frac{\pi x'}{a}\right)$ and again by injectivity on $\left(0,\frac{a}{2}\right]$ we deduce $x=x'$ as desired.
\end{proof}

We next prove the main result of this section.

\begin{theorem}\label{thm:rectangle}
Assume $\mathcal R=(0,a)\times (0,b)$ is a rectangular domain with $a<b$. Let $u_k(x,y)=\mathbb{E}^{(x,y)}[\tau_{\mathcal R}^k]$ denote the exit time moments for $k\geq 1$. Given positive real numbers $c_1, c_2, ...$, the intersection of the level curves $u_k(x,y)=c_k$ is either empty, or consists of one, two, or four points.
\end{theorem}
\begin{proof}
Suppose that $(x,y), (x', y') \in \mathcal R$ satisfy $u_k(x,y)=u_k(x',y')=c_k$ for all $k\geq 1$. First suppose $(x,y), (x', y') \in \mathcal R'$. Courtesy of Proposition \ref{prop:eigendisc}, let $\lambda_{\xi,\nu}$ be the least eigenvalue such that 
    \begin{equation}
        \phi_{\xi,\nu}(x,y) \ne \phi_{\xi,\nu}(x',y').
    \end{equation}
Then we know that either $\lambda_{\xi,\nu}=\lambda_{1,1}$ or $\lambda_{\xi,\nu}=\lambda_{1,3}$ and since $b>a$ there are no other positive odd indices $(m,n)$ such that $\lambda_{\xi,\nu}=\lambda_{m,n}$. Then for each $k$,
\begin{align}
&u_k(x,y)-u_k(x',y') \nonumber \\
&=\frac{16k!}{\pi^2\lambda^k_{\xi,\nu}\xi\nu}[\phi_{\xi,\nu}(x,y) - \phi_{\xi,\nu}(x',y')]+ 
\sum_{\substack{\lambda_{m,n} > \lambda_{\xi,\nu} \\ \text{$m,n$ odd}}}\frac{16k!} {\pi^2\lambda^k_{m,n}mn} [\phi_{m,n}(x,y) - \phi_{m,n}(x',y')]\label{eq:diffbig}\\
&=0. \nonumber
\end{align}
Again, courtesy of Proposition \ref{prop:eigendisc}, the first term in \eqref{eq:diffbig} is nonzero. Dividing out by that first term gives the equation
\begin{equation}\label{eq:ratio}
        0  = 1 +  \sum_{\substack{\lambda_{m,n} > \lambda_{\xi,\nu} \\ \text{$m,n$ odd}}} \Big(\frac{\lambda_{\xi,\nu}}{\lambda_{m,n}}\Big)^k C_{m,n},
\end{equation}
where
\[
C_{m,n} = \frac{\xi\nu}{mn} \frac{[\phi_{m,n}(x,y) - \phi_{m,n}(x',y')]}{[\phi_{\xi,\nu}(x,y) - \phi_{\xi,\nu}(x',y')]}.
\]
Denote
\[
S_k=\sum_{\substack{\lambda_{m,n} > \lambda_{\xi,\nu} \\ \text{$m,n$ odd}}} \Big(\frac{\lambda_{\xi,\nu}}{\lambda_{m,n}}\Big)^k
\]
for $k\geq 2$. We claim that the series defining $S_k$ converges for all $k\geq 2$. It suffices to show that $S_2$ converges. Note that
\begin{align*}
S_2&=\sum_{\substack{\lambda_{m,n} > \lambda_{\xi,\nu} \\ \text{$m,n$ odd}}}\Big(\frac{\lambda_{\xi,\nu}}{\lambda_{m,n}}\Big)^2\\
&=\lambda_{\xi,\nu}^2\sum_{\substack{\lambda_{m,n} > \lambda_{\xi,\nu} \\ \text{$m,n$ odd}}}\left(\frac{1}{\left(\frac{m\pi}{a}\right)^2+\left(\frac{n\pi}{b}\right)^2}\right)^2\\
&\leq \lambda_{\xi,\nu}^2\frac{b^4}{\pi^4}\sum_{\substack{\lambda_{m,n} > \lambda_{\xi,\nu} \\ \text{$m,n$ odd}}}\left(\frac{1}{m^2+n^2}\right)^2
\end{align*}
and so it suffices to show that $\displaystyle \sum_{m,n=1}^{\infty}\left(\frac{1}{m^2+n^2}\right)^2$ converges. But
\[
\sum_{m,n=1}^{\infty}\left(\frac{1}{m^2+n^2}\right)^2\leq \sum_{m,n=1}^{\infty}\frac{1}{4}\frac{1}{m^2}\frac{1}{n^2}=\frac{\pi^4}{144}<\infty
\]
so $S_2$ indeed converges. Having shown that $S_k$ converges for $k\geq 2$, we note that for such $k$,
\[
S_{k+1}=\sum_{\substack{\lambda_{m,n} > \lambda_{\xi,\nu} \\ \text{$m,n$ odd}}}\Big(\frac{\lambda_{\xi,\nu}}{\lambda_{m,n}}\Big)\Big(\frac{\lambda_{\xi,\nu}}{\lambda_{m,n}}\Big)^k\leq\Big(\frac{\lambda_{\xi,\nu}}{\lambda_{m',n'}}\Big)\sum_{\substack{\lambda_{m,n} > \lambda_{\xi,\nu} \\ \text{$m,n$ odd}}}\Big(\frac{\lambda_{\xi,\nu}}{\lambda_{m,n}}\Big)^k = \Big(\frac{\lambda_{\xi,\nu}}{\lambda_{m',n'}}\Big)S_k,
\]
where $\lambda_{m',n'}$ is the first eigenvalue with $m',n'$ odd and $\lambda_{m',n'}>\lambda_{\xi,\nu}$. Letting $k\to \infty$ in the inequality immediately above, it must be the case that $\lim_{k\to \infty}S_k=0$. Thus
\begin{align*}
\left|\sum_{\substack{\lambda_{m,n} > \lambda_{\xi,\nu} \\ \text{$m,n$ odd}}} \Big(\frac{\lambda_{\xi,\nu}}{\lambda_{m,n}}\Big)^k C_{m,n}\right|&\leq \sum_{\substack{\lambda_{m,n} > \lambda_{\xi,\nu} \\ \text{$m,n$ odd}}} \Big(\frac{\lambda_{\xi,\nu}}{\lambda_{m,n}}\Big)^k |C_{m,n}|\\
&\leq \frac{2\xi\nu}{\left|\phi_{\xi,\nu}(x,y) - \phi_{\xi,\nu}(x',y')\right|}\sum_{\substack{\lambda_{m,n} > \lambda_{\xi,\nu} \\ \text{$m,n$ odd}}}\Big(\frac{\lambda_{\xi,\nu}}{\lambda_{m,n}}\Big)^k\\
&=\frac{2\xi\nu}{\left|\phi_{\xi,\nu}(x,y) - \phi_{\xi,\nu}(x',y')\right|} S_k,
\end{align*} where the second inequality follows since $\phi_{m,n}$ are products of sine functions.
Since we have already shown that $\lim_{k\to \infty}S_k=0$, letting $k\to \infty$ in \eqref{eq:ratio} yields the contradiction $0=1$. We have thus shown that if $u_k(x,y)=u_k(x',y')$ for each $k\geq 1$ and $(x,y),(x',y')\in \mathcal R'$ then $(x,y)=(x',y')$.

Next, suppose that $(x,y),(x',y')\in \mathcal R$ and that $u_k(x,y)=u_k(x',y')$ for each $k\geq 1$. It can be verified that the exit time moments $u_k$ are invariant under reflection through the vertical line $x=\frac{a}{2}$ and horizontal reflection through the line $y=\frac{b}{2}$. Choose  points $(x'',y''),(x''',y''')\in \mathcal R'$ obtained from $(x,y)$ and $(x',y')$, respectively, by vertical reflection in $x=\frac{a}{2}$ and/or horizontal reflection through the line $y=\frac{b}{2}$. Then $u_k(x'',y'')=u_k(x''',y''')$ for each $k\geq 1$ and we deduce that $(x'',y'')=(x''',y''')$. It follows that the points $(x,y)$ and $(x',y')$ are obtained from each other by vertical reflection through $x=\frac{a}{2}$ and/or horizontal reflection through the line $y=\frac{b}{2}$. If $x''<\frac{a}{2}$ and $y''<\frac{b}{2}$, then the intersection of the level curves $u_k=c_k$ consists of four points. If $x''=\frac{a}{2}$ and $y''<\frac{b}{2}$, then the intersection of the level curves $u_k=c_k$ consists of two points. If $x''<\frac{a}{2}$ and $y''=\frac{b}{2}$, the intersection of the level curves $u_k=c_k$ again consists of two points. Finally, if $x''=\frac{a}{2}$ and $y''=\frac{b}{2}$, the intersection of the level curves $u_k=c_k$ consists of one point.
\end{proof}

As before, we have the following corollary on location detection.

\begin{corollary}\label{cor:locdetrect}
Consider a rectangular domain $\mathcal R$ with $b>a$. Given values of the full sequence of exit time moments at some unknown point $(x,y)$ in $\mathcal R$, one can detect the location $(x,y)$ up to symmetry.
\end{corollary}

As noted in the introduction, there are rectangles in which the level curves of the first two exit time moments intersect at more than four points.

\begin{remark}\label{rmk:numevid}
Let $\mathcal R=(0,1)\times (0,2)$ and let $u_k(x,y)=\mathbb{E}^{(x,y)}[\tau_{\mathcal R}^k]$ for $k=1,2$. Taking $c_1=0.015$ and $c_2=0.0015$, Figure \ref{fig:rectlevelcurv} below shows the level curves $u_1=c_1$ and $u_2=c_2$ intersect at the distinct points $P$ and $Q$ in $\mathcal R'=(0,0.5)\times (0,1)$. By symmetry, the full level curves intersect at $8$ distinct points in $\mathcal R$.

\begin{figure}[H]
    \centering
\begin{tikzpicture}[scale=.62]
  \begin{axis}[
    width=11cm,
    height=15cm,
    xmin=0, xmax=0.5,
    ymin=0, ymax=1.0,
    xtick distance=0.1,
    xticklabel style={/pgf/number format/fixed, /pgf/number format/precision=1},
    xlabel={$x$},
    ylabel={$y$},
    axis lines=box,
    grid=major,
    grid style={dashed, opacity=0.3},
    title={Level Curves of $u_1$ and $u_2$ in $\mathcal{R}' = (0, 0.5) \times (0, 1.0)$},
    legend style={at={(0.95,0.95)}, anchor=north east, draw=white!80!black}
  ]

    \addplot[
      blue, 
      thick, 
      smooth
    ] coordinates {
      (0.4900,0.0430) (0.4783,0.0431) (0.4666,0.0432) (0.4548,0.0433) (0.4431,0.0434) 
      (0.4314,0.0436) (0.4197,0.0439) (0.4079,0.0441) (0.3962,0.0444) (0.3845,0.0448) 
      (0.3728,0.0452) (0.3610,0.0457) (0.3493,0.0462) (0.3376,0.0468) (0.3259,0.0474) 
      (0.3141,0.0481) (0.3024,0.0489) (0.2907,0.0498) (0.2790,0.0508) (0.2672,0.0519) 
      (0.2555,0.0531) (0.2438,0.0545) (0.2321,0.0560) (0.2203,0.0577) (0.2086,0.0597) 
      (0.1969,0.0619) (0.1852,0.0645) (0.1734,0.0676) (0.1617,0.0711) (0.1500,0.0753) 
      (0.1391,0.0800) (0.1140,0.0956) (0.0979,0.1112) (0.0868,0.1267) (0.0785,0.1423) 
      (0.0722,0.1579) (0.0672,0.1735) (0.0632,0.1890) (0.0599,0.2046) (0.0572,0.2202) 
      (0.0548,0.2358) (0.0527,0.2513) (0.0509,0.2669) (0.0492,0.2825) (0.0478,0.2981) 
      (0.0466,0.3136) (0.0455,0.3292) (0.0445,0.3448) (0.0436,0.3604) (0.0427,0.3759) 
      (0.0419,0.3915) (0.0412,0.4071) (0.0406,0.4227) (0.0400,0.4383) (0.0395,0.4538) 
      (0.0390,0.4694) (0.0386,0.4850) (0.0381,0.5006) (0.0377,0.5161) (0.0374,0.5317) 
      (0.0370,0.5473) (0.0367,0.5629) (0.0365,0.5784) (0.0362,0.5940) (0.0359,0.6096) 
      (0.0357,0.6252) (0.0355,0.6407) (0.0353,0.6563) (0.0351,0.6719) (0.0349,0.6875) 
      (0.0348,0.7031) (0.0346,0.7186) (0.0345,0.7342) (0.0343,0.7498) (0.0342,0.7654) 
      (0.0341,0.7809) (0.0340,0.7965) (0.0339,0.8121) (0.0339,0.8277) (0.0338,0.8432) 
      (0.0337,0.8588) (0.0336,0.8744) (0.0336,0.8900) (0.0335,0.9055) (0.0335,0.9211) 
      (0.0335,0.9367) (0.0335,0.9523) (0.0334,0.9678) (0.0334,0.9834) (0.0334,0.9990)
    };
    \addlegendentry{$u_1 = 0.015$}

    \addplot[
      red, 
      thick, 
      smooth
    ] coordinates {
      (0.4900,0.0381) (0.4783,0.0381) (0.4666,0.0383) (0.4548,0.0384) (0.4431,0.0386) 
      (0.4314,0.0389) (0.4197,0.0392) (0.4079,0.0396) (0.3962,0.0400) (0.3845,0.0405) 
      (0.3728,0.0411) (0.3610,0.0417) (0.3493,0.0424) (0.3376,0.0432) (0.3259,0.0441) 
      (0.3141,0.0451) (0.3024,0.0462) (0.2907,0.0474) (0.2790,0.0488) (0.2672,0.0503) 
      (0.2555,0.0519) (0.2438,0.0538) (0.2321,0.0559) (0.2203,0.0582) (0.2086,0.0609) 
      (0.1969,0.0639) (0.1852,0.0674) (0.1734,0.0713) (0.1617,0.0759) (0.1526,0.0800) 
      (0.1262,0.0956) (0.1080,0.1112) (0.0947,0.1267) (0.0845,0.1423) (0.0764,0.1579) 
      (0.0699,0.1735) (0.0645,0.1890) (0.0600,0.2046) (0.0562,0.2202) (0.0529,0.2358) 
      (0.0500,0.2513) (0.0475,0.2669) (0.0453,0.2825) (0.0433,0.2981) (0.0416,0.3136) 
      (0.0400,0.3292) (0.0386,0.3448) (0.0373,0.3604) (0.0362,0.3759) (0.0351,0.3915) 
      (0.0341,0.4071) (0.0332,0.4227) (0.0324,0.4383) (0.0317,0.4538) (0.0310,0.4694) 
      (0.0303,0.4850) (0.0297,0.5006) (0.0292,0.5161) (0.0287,0.5317) (0.0282,0.5473) 
      (0.0278,0.5629) (0.0273,0.5784) (0.0270,0.5940) (0.0266,0.6096) (0.0263,0.6252) 
      (0.0260,0.6407) (0.0257,0.6563) (0.0254,0.6719) (0.0252,0.6875) (0.0249,0.7031) 
      (0.0247,0.7186) (0.0245,0.7342) (0.0244,0.7498) (0.0242,0.7654) (0.0240,0.7809) 
      (0.0239,0.7965) (0.0238,0.8121) (0.0236,0.8277) (0.0235,0.8432) (0.0234,0.8588) 
      (0.0234,0.8744) (0.0233,0.8900) (0.0232,0.9055) (0.0232,0.9211) (0.0231,0.9367) 
      (0.0231,0.9523) (0.0231,0.9678) (0.0231,0.9834) (0.0230,0.9990)
    };
    \addlegendentry{$u_2 = 0.0015$}

    \node[
      circle, 
      fill=black, 
      inner sep=1.8pt, 
      label={right:{\textbf{P}\,$\approx (0.059, 0.207)$}}
    ] at (axis cs:0.059304, 0.207350) {};

    \node[
      circle, 
      fill=black, 
      inner sep=1.8pt, 
      label={above right:{\textbf{Q}\,$\approx (0.229, 0.056)$}}
    ] at (axis cs:0.229034, 0.056465) {};

  \end{axis}
\end{tikzpicture}
\caption{Level curves of the first two exit time moments plotted in $\mathcal R'$.}
\label{fig:rectlevelcurv}
\end{figure}
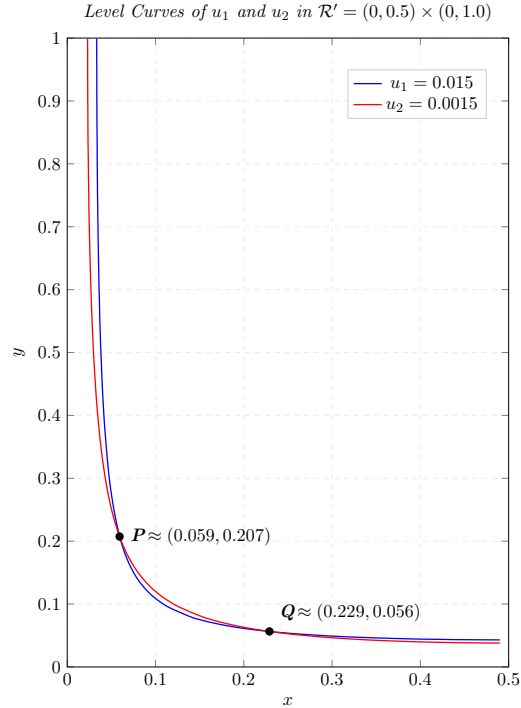
\end{remark}

\section{General Results}\label{secg:gen}

In this section we prove location detection results for general domains. Let $D\subseteq \mathbb{R}^2$ denote a bounded Lipschitz domain. We enumerate the eigenvalues of the Dirichlet eigenvalue problem
 \[
       \left\{
        \begin{aligned}
            -\Delta \phi &= \lambda\phi \text{ in } D, \\
            \phi &= 0  \text{ on } \partial D, 
        \end{aligned}
        \right.
\]
as
\[
0<\lambda_1<\lambda_2\leq \lambda_3\leq \cdots
\]
and let $\phi_1, \phi_2,\ldots$ denote an associated complete orthonormal set of eigenfunctions. For each eigenvalue $\lambda$, let $E_{\lambda}$ denote the eigenspace associated with $\lambda$ and let $\Phi_{\lambda}$ denote the orthogonal projection of the constant function $1$ onto the eigenspace $E_{\lambda}$:
\begin{equation}\label{eq:projdef}
\Phi_{\lambda}=\textup{Proj}_{E_{\lambda}}1=\sum_{\lambda_j=\lambda}\langle 1,\phi_j\rangle \phi_j,
\end{equation}
where $\langle \cdot, \cdot \rangle$ is the standard $L^2$-inner product. Let
\[
\textup{Spec}^{\ast}(D)=\{\lambda_j : \Phi_{\lambda_j}\neq 0\}
\]
denote the collection of eigenvalues for which the associated eigenspace is not orthogonal to constants. Consequently, we think of $\textup{Spec}^{\ast}(D)$ as those eigenvalues that are ``visible'' to constants. The following result says that two points in $D$ have the same exit time moments precisely when their visible spectral projections agree. In particular, whenever the visible spectral projections separate points, the exit time moments uniquely determine location.

\begin{theorem}\label{th:genth1}
Let $D$ be a domain as above and let $u_k(x,y)=\mathbb{E}^{(x,y)}[\tau_D^k]$ for $k\geq 1$ denote the exit time moments from $D$. Let $p,q\in D$ be distinct points. Then $\Phi_{\mu}(p)=\Phi_{\mu}(q)$ for every $\mu \in \textup{Spec}^{\ast}(D)$ if and only if $u_k(p)=u_k(q)$ for every $k\geq 1$. Consequently, if the visible spectral projections separate points in $D$, and the level curves of the exit time moments intersect, then their intersection is a single point.
\end{theorem}

\begin{proof}
First, assume that $p, q\in D$ are distinct and $u_k(p)=u_k(q)$ for every $k\geq 1$ but $\Phi_{\mu}(p)\neq \Phi_{\mu}(q)$ for some $\mu \in \textup{Spec}^{\ast}(D)$. The eigenfunction expansion of the constant function $1$ is
\[
1=\sum_{j=1}^{\infty}b_j{\phi_j},
\]
where $b_j=\langle 1, \phi_j\rangle$. The exit time moment $u_k$ thus has eigenfunction expansion
\[
u_k=\sum_{j=1}^{\infty}\frac{k!b_j}{\lambda_j^k}\phi_j.
\]
If we enumerate the values of $\textup{Spec}^{\ast}(D)$ as $\mu_1<\mu_2<\cdots$, then from \eqref{eq:projdef} we see
\begin{equation}\label{eq:momspecstar}
u_k=\sum_{j=1}^{\infty}\frac{k!}{\mu_j^k}\Phi_{\mu_j}.
\end{equation}
By assumption, there is some minimal $\mu_N\in \textup{Spec}^{\ast}(D)$ such that $\Phi_{\mu_N}(p)\neq \Phi_{\mu_N}(q)$. From \eqref{eq:momspecstar}, we see
\[
0=u_k(p)-u_k(q)=\frac{k!}{\mu_N^k}\left( \Phi_{\mu_N}(p)-\Phi_{\mu_N}(q)\right)+\sum_{j=N+1}^{\infty}\frac{k!}{\mu_j^k}\left( \Phi_{\mu_j}(p)-\Phi_{\mu_j}(q)\right).
\]
Dividing out by this first nonzero term gives
\begin{equation}\label{eq:divoutgen}
0=1+\sum_{j=N+1}^{\infty}\left(\frac{\mu_N}{\mu_j}\right)^k\frac{\Phi_{\mu_j}(p)-\Phi_{\mu_j}(q)}{\Phi_{\mu_N}(p)-\Phi_{\mu_N}(q)}.
\end{equation}
Note that
\[
\Phi_{\mu_j}(p)-\Phi_{\mu_j}(q)=\sum_{\lambda_\ell=\mu_j}b_{\ell}(\phi_{\ell}(p)-\phi_{\ell}(q))
\]
and so \eqref{eq:divoutgen} becomes
\[
0=1+\frac{1}{\Phi_{\mu_N}(p)-\Phi_{\mu_N}(q)}\sum_{\lambda_j>\mu_N}\left(\frac{\mu_N}{\lambda_j}\right)^kb_j(\phi_j(p)-\phi_{j}(q)).
\]
Denote
\[
S_k=\frac{1}{\Phi_{\mu_N}(p)-\Phi_{\mu_N}(q)}\sum_{\lambda_j>\mu_N}\left(\frac{\mu_N}{\lambda_j}\right)^kb_j(\phi_j(p)-\phi_{j}(q)).
\]
Note that by Cauchy-Schwarz,
\begin{equation}\label{eq:bjbound}
|b_j|=\left|\int_D\phi_j\,dx\,dy\right|\leq \left(\int_D 1\, dx \,dy \right)^{\frac{1}{2}}\left( \int_D \phi_j^2 \,dx\,dy\right)^{\frac{1}{2}}=|D|^{\frac{1}{2}},
\end{equation}
where $|D|$ is the area of $D$. And by standard estimates (see \cite{BHV}, for instance) there exists some constant $C$ independent of the eigenvalues where
\begin{equation}\label{eq:sogge}
\|\phi_j\|_{\infty}\leq C\lambda_j^{\frac{1}{2}}.
\end{equation}
Combining \eqref{eq:bjbound} and \eqref{eq:sogge}, if $k\geq 2$ we see
\begin{align}
|S_k|&\leq \frac{1}{|\Phi_{\mu_N}(p)-\Phi_{\mu_N}(q)|}\sum_{\lambda_j>\mu_N}\left(\frac{\mu_N}{\lambda_j}\right)^k|b_j|\,|\phi_j(p)-\phi_{j}(q)| \nonumber \\
&\leq \frac{2C|D|^{\frac{1}{2}}}{|\Phi_{\mu_N}(p)-\Phi_{\mu_N}(q)|}\sum_{\lambda_j>\mu_N}\left(\frac{\mu_N}{\lambda_j}\right)^k\lambda_j^{\frac{1}{2}} \nonumber \\
&= \frac{2C|D|^{\frac{1}{2}}\mu_N^2}{|\Phi_{\mu_N}(p)-\Phi_{\mu_N}(q)|}\sum_{\lambda_j>\mu_N}\left(\frac{\mu_N}{\lambda_j}\right)^{k-2}\frac{1}{\lambda_j^{\frac{3}{2}}} \nonumber \\
&\leq \frac{2C|D|^{\frac{1}{2}}\mu_N^2}{|\Phi_{\mu_N}(p)-\Phi_{\mu_N}(q)|}\left(\frac{\mu_N}{\lambda^\ast}\right)^{k-2}\sum_{\lambda_j>\mu_N}\frac{1}{\lambda_j^{\frac{3}{2}}}, \label{eq:weyl32}
\end{align}
where $\lambda^{\ast}$ is the first Dirichlet eigenvalue that exceeds $\mu_N$. By the Weyl asymptotics, 
\[
\lambda_j \sim \frac{4\pi}{|D|}j
\]
and so the series in \eqref{eq:weyl32} converges. Since $\frac{\mu_N}{\lambda^{\ast}}<1$, we see that each $S_k$ converges absolutely for $k\geq 2$ and $S_k\to 0$ as $k\to \infty$. Thus letting $k\to \infty$ in \eqref{eq:divoutgen} yields the contradiction $0=1$.

Next assume that $p,q\in D$ are distinct and $\Phi_{\mu}(p)=\Phi_{\mu}(q)$ for every $\mu \in \textup{Spec}^{\ast}(D)$. Then \eqref{eq:momspecstar} immediately gives that $u_k(p)=u_k(q)$ for every $k\geq 1$.
\end{proof}

A similar result holds for domains with two axes of symmetry. More precisely we have the following.

\begin{theorem}\label{th:genth2}
Let $D\subseteq \mathbb{R}^2$ denote a bounded Lipschitz domain that is symmetric with respect to the $x$- and $y$-axes and let $u_k(x,y)=\mathbb{E}^{(x,y)}[\tau_D^k]$ for $k\geq 1$ denote the exit time moments from $D$. Denote
\[
D^+=\{(x,y)\in D:x,y\geq 0\}
\]
and assume that whenever $p,q\in D^+$ are distinct, there is some $\mu \in \textup{Spec}^{\ast}(D)$ where $\Phi_{\mu}(p)\neq \Phi_{\mu}(q)$. Then distinct points $p,q\in D$ have the same exit time moments if and only if $p$ and $q$ are obtained from each other by reflection across the $x$- and/or $y$-axes. Consequently, if $c_k$ for $k\geq 1$ are positive numbers such that the level curves $u_k=c_k$ intersect, then their intersection is either one, two, or four points determined by the values $c_k$.
\end{theorem}

\begin{proof}
Suppose that $p,q\in D$ are distinct and that $u_k(p)=u_k(q)$ for each $k\geq 1$. Since the $u_k$ are invariant under reflection across the $x$- and $y$-axes, choose  points $r,s\in D^+$ obtained from $p$ and $q$, respectively, by vertical reflection across the $y$-axis and/or horizontal reflection across the $x$-axis. Then $u_k(r)=u_k(s)$ for each $k\geq 1$. The proof of Theorem \ref{th:genth1} applied to $D^+$ gives that $\Phi_{\mu}(r)=\Phi_{\mu}(s)$ for every $\mu \in \textup{Spec}^{\ast}(D)$. It follows that the points $p$ and $q$ are obtained from each other by vertical reflection across the $y$-axis and/or horizontal reflection across the $x$-axis. Conversely if $p$ and $q$ are obtained from each other by reflection across the $x$- and/or $y$-axes, then the exit time moments have to agree at $p$ and $q$ because the moments themselves are invariant under those reflections.

Next let $p$ be a point in the intersection of all of the level curves of the exit time moments. If $p$ does not lie on either coordinate axis, then the intersection of the level curves $u_k=c_k$ consists of four points. If $p$ lies on the $y$-axis but not the $x$-axis, then the intersection of the level curves $u_k=c_k$ consists of two points. If $p$ lies on the $x$-axis but not the $y$-axis, the intersection of the level curves $u_k=c_k$ consists again of two points. Finally, if $p$ lies on both coordinate axes, the intersection of the level curves $u_k=c_k$ consists of one point.

\end{proof}

Our last remark illustrates that the visible spectral projections are the key tools needed to detect location in rectangular domains.

\begin{remark}
Let $D=\left(-\frac{a}{2}, \frac{a}{2}\right) \times \left(-\frac{b}{2}, \frac{b}{2}\right)$ with $b>a$. Then the Dirichlet eigenvalues are
\[
\lambda_{m,n}=\left( \frac{m\pi}{a} \right)^2+\left( \frac{n\pi}{b} \right)^2
\]
and the $L^2$-normalized eigenfunctions are
\[
\phi_{m,n}(x,y)=\frac{2}{\sqrt{ab}} \sin\left(\frac{m\pi}{a}\left(x+\frac{a}{2}\right)\right) \sin\left(\frac{n\pi}{b}\left(y+\frac{b}{2}\right)\right)
\]
for $m,n\geq 1$. In general
\[
\langle 1, \phi_{m,n}\rangle=\int_D\phi_{m,n}\,dx\,dy=\frac{2\sqrt{ab}}{mn\pi^2}(1-(-1)^m)(1-(-1)^n)
\]
and so $\langle 1, \phi_{m,n}\rangle =0$ if either $m$ or $n$ is even. It follows that for $\lambda \in \textup{Spec}^{\ast}(D)$, it is necessary that there exist $m$ and $n$ both odd where $\lambda=\lambda_{m,n}$. The eigenvalue $\lambda_{1,1}$ is simple and its eigenfunction is positive in $D$ so $\lambda_{1,1}\in \textup{Spec}^{\ast}(D)$. We have
\[
\Phi_{\lambda_{1,1}}=\langle 1, \phi_{1,1}\rangle \phi_{1,1}=\frac{16}{\pi^2}\sin\left(\frac{\pi}{a}\left(x+\frac{a}{2}\right)\right) \sin\left(\frac{\pi}{b}\left(y+\frac{b}{2}\right)\right).
\]
Also, $\lambda_{1,3}\in \textup{Spec}^{\ast}(D)$ because there are no other eigenvalues $\lambda_{m,n}$ with $m$ and $n$ odd where $\lambda_{1,3}=\lambda_{m,n}$ and as noted above, if $\lambda_{1,3}=\lambda_{m,n}$ for some $m$ and $n$, at least one of which is even, then $\langle 1, \phi_{m,n}\rangle =0$. It follows that
\[
\Phi_{\lambda_{1,3}}=\langle 1, \phi_{1,3}\rangle \phi_{1,3}=\frac{16}{3\pi^2}\sin\left(\frac{\pi}{a}\left(x+\frac{a}{2}\right)\right) \sin\left(\frac{3\pi}{b}\left(y+\frac{b}{2}\right)\right).
\]
Now let $\mathcal R=(0,a)\times (0,b)$. With the notation from Section \ref{sect:rect}, $\mathcal R'=\left(0,\frac{a}{2}\right] \times \left(0, \frac{b}{2}\right]$. The coordinate map $L:\mathbb{R}^2 \to \mathbb{R}^2$ defined by
\[
L(x,y)=\left(\frac{a}{2}-x , \frac{b}{2}-y \right)
\]
is a bijection from $D$ to $\mathcal R$ and takes $D^+$ to $\mathcal R'$. Given distinct points $(x,y),(x',y')\in D^+$, the points $L(x,y)$ and $L(x',y')$ are distinct in $\mathcal R'$. In Proposition \ref{prop:eigendisc}, we showed that either
\begin{equation}\label{eq:rmktrig1}
\sin \left(\frac{\pi}{a}\left(\frac{a}{2}-x\right)\right)\sin\left(\frac{\pi}{b}\left(\frac{b}{2}-y \right)\right) \neq \sin\left(\frac{\pi}{a}\left(\frac{a}{2}-x'\right)\right)\sin\left(\frac{\pi }{b}\left(\frac{b}{2}-y' \right)\right)
\end{equation}
or
\begin{equation}\label{eq:rmktrig2}
\sin \left(\frac{\pi}{a}\left(\frac{a}{2}-x\right)\right)\sin\left(\frac{3\pi}{b}\left(\frac{b}{2}-y \right)\right) \neq \sin\left(\frac{\pi}{a}\left(\frac{a}{2}-x'\right)\right)\sin\left(\frac{3\pi }{b}\left(\frac{b}{2}-y' \right)\right).
\end{equation}
Using the trig identities $\sin \left(\frac{\pi}{2}-\theta \right)=\sin \left(\frac{\pi}{2}+\theta \right)$ and $\sin \left(\frac{3\pi}{2}-\theta \right)=\sin \left(\frac{3\pi}{2}+\theta \right)$ in \eqref{eq:rmktrig1} and \eqref{eq:rmktrig2}, it must be that either
\[
\sin\left(\frac{\pi}{a}\left(x+\frac{a}{2}\right)\right) \sin\left(\frac{\pi}{b}\left(y+\frac{b}{2}\right)\right)  \neq \sin\left(\frac{\pi}{a}\left(x'+\frac{a}{2}\right)\right) \sin\left(\frac{\pi}{b}\left(y'+\frac{b}{2}\right)\right)
\]
or
\[
\sin\left(\frac{\pi}{a}\left(x+\frac{a}{2}\right)\right) \sin\left(\frac{3\pi}{b}\left(y+\frac{b}{2}\right)\right)  \neq \sin\left(\frac{\pi}{a}\left(x'+\frac{a}{2}\right)\right) \sin\left(\frac{3\pi}{b}\left(y'+\frac{b}{2}\right)\right).
\]
In other words, $\Phi_{\lambda_{1,1}}(x,y)\neq \Phi_{\lambda_{1,1}}(x',y')$ or $\Phi_{\lambda_{1,3}}(x,y)\neq \Phi_{\lambda_{1,3}}(x',y')$. That is, points in $D^+$ are separated by either $\Phi_{\lambda_{1,1}}$ or $\Phi_{\lambda_{1,3}}$.
\end{remark}

\subsection*{Artificial Intelligence Acknowledgement}
    We acknowledge the use of ChatGPT and Google Gemini for its use in generating LaTeX/Mathematica code for our figures. Theorem \ref{th:genth1} in Section \ref{secg:gen} was developed through conversations with ChatGPT.

\end{document}